\documentclass[twoside]{amsart}
\usepackage{mathrsfs}
\usepackage{amsfonts}
\usepackage{amssymb}
\usepackage{amssymb}
\usepackage{amssymb}
\usepackage{amssymb}
\usepackage{amssymb}
\usepackage{amssymb}
\usepackage{amssymb}
\usepackage{amssymb}
\usepackage{amsmath}
\usepackage[all]{xy}

\usepackage{lastpage}

\RequirePackage{amsmath} \RequirePackage{amssymb}
\usepackage{amscd,latexsym,amsthm,amsfonts,amssymb,amsmath,amsxtra}
\usepackage[colorlinks=true, urlcolor=blue,bookmarks=true,bookmarksopen=true,
citecolor=blue,hypertex]{hyperref}

    \newcommand{\BC}{{\mathbb {C}}}

     \newcommand{\BR}{{\mathbb {R}}}

      \newcommand{\bB}{{\textbf {B}}}

     \newcommand{\CB}{{\mathcal {B}}}
    \newcommand{\CC}{{\mathcal {C}}}

    \newcommand{\CO}{{\mathcal {O}}} \newcommand{\CP}{{\mathcal {P}}}
     
    \newcommand{\CS}{{\mathcal {S}}} \newcommand{\CT}{{\mathcal {T}}}
     \newcommand{\CV}{{\mathcal {V}}}
    \newcommand{\CW}{{\mathcal {W}}}

       \newcommand{\fp}{{\mathfrak{p}}}

       \newcommand{\fD}{{\mathfrak{D}}}

    \newcommand{\RG}{{\mathrm {G}}}

    \newcommand{\RU}{{\mathrm {U}}}

    \newcommand{\diag}{{\mathrm {diag}}}

     \newcommand{\Nm}{{\mathrm {Nm}}}
    \newcommand{\lenth}{{\mathrm {\lenth}}}

    \newcommand{\End}{{\mathrm{End}}} 
    \newcommand{\Gal}{{\mathrm{Gal}}} \newcommand{\GL}{{\mathrm{GL}}}
    \newcommand{\Hom}{{\mathrm{Hom}}} 
    \newcommand{\Ind}{{\mathrm{Ind}}} \newcommand{\ind}{{\mathrm{ind}}}

    \newcommand{\Ker}{{\mathrm{Ker}}}

    \renewcommand{\Re}{{\mathrm{Re}}} 
    \newcommand{\Res}{{\mathrm{Res}}}

 \newcommand{\Vol}{{\mathrm{Vol}}}

\newcommand{\Supp}{\mathrm{Supp}}\newcommand{\supp}{\mathrm{supp}}

 \newcommand{\SL}{{\mathrm{SL}}}

 \newcommand{\tr}{{\mathrm{tr}}}
  
\newcommand{\vol}{{\mathrm{vol}}}

  \newcommand{\Mp}{{\mathrm{Mp}}}
 \newcommand{\Sp}{{\mathrm{Sp}}}

    \newcommand{\pair}[1]{\langle {#1} \rangle}
    \newcommand{\wpair}[1]{\left\{{#1}\right\}}

    \newcommand{\incl}{\hookrightarrow}
    
     \newcommand{\ra}{\rightarrow}

    \theoremstyle{plain}

    \newtheorem{thm}{Theorem}[section] \newtheorem{cor}[thm]{Corollary}
    \newtheorem{lem}[thm]{Lemma}  \newtheorem{prop}[thm]{Proposition}

 \newtheorem{rem}[thm]{Remark}
    \numberwithin{equation}{section}

\usepackage[top=1in, bottom=1in, left=1.25in, right=1.25in]{geometry}

\title{A Local Converse Theorem for $\RU(1,1)$}
\author{Qing Zhang}

\begin{document}

\maketitle
\begin{abstract}
In this paper, we define a $\gamma$-factor for generic representations of $\RU(1,1)\times \Res_{E/F}(\GL_1)$ and prove a local converse theorem for $\RU(1,1)$ using the $\gamma$-factor we defined. We also give a new proof of the local converse theorem for $\GL_2$ using a $\gamma$-factor of $\GL_2\times \GL_2$ type which was originally defined by Jacquet in \cite{J}.
\end{abstract}

\setcounter{tocdepth}{1}
\tableofcontents

\section*{Introduction}
Let $E/F$ be a quadratic extension of number fields and $H=\RU(1,1)$ be the quasi-split unitary group of rank $2$ defined for $E/F$. In $\S$7 of \cite{GRS}, Gelbart, Rogawski and Soudry defined a global zeta integral for $\RU(1,1)\times \Res_{E/F}(\GL_1)$, proved that it is Eulerian and did the unramified calculation. This global zeta integral is an analogue of the integral defined for $\Sp_{2n}\times \GL_n$ case considered by Gelbart and Piatetski-Shapiro in \cite{GPS} (the method C in \cite{GPS}).

In this paper, we will show the existence of the local $\gamma$-factors for $\RU(1,1)\times \Res_{E/F}(\GL_1)$ and obtain a local converse theorem for $\RU(1,1)$ except when $E/F$ is ramified and the residue characteristic of $F$ is $2$. In the following, we describe our results in more details. For the local group $H(F_v)$, there are 2 cases to consider. At a non-split place $v$ of $F$, we have $H(F_v)=\RU(1,1)(F_v)$, where $\RU(1,1)$ is defined for the local extension $E_w/F_v$ where $w$ is the unique place of $E$ above $v$. At a split place $v$ of $F$, we have $H(F_v)=\GL_2(F_v)$. In this paper, we consider the two local cases separately.

We first consider the local $\RU(1,1)$ case.  Assume that $E/F$ is a quadratic extension of $p$-adic fields and $H=\RU(1,1)(F)$. We fix an additive character $\psi$ of $F$, which is also viewed as a character of the upper triangular unipotent subgroup $N$ of $H$. We also fix an element $\kappa\in F^\times -\Nm_{E/F}(E^\times)$ and consider the character $\psi_\kappa$ defined by $\psi_{\kappa}(x)=\psi(\kappa x)$. Let $\pi$ be an infinite dimensional irreducible smooth representation of $H$. Then $H$ is either $\psi$-generic or $\psi_\kappa$-generic, see $\S$1. Let $\chi$ be a character $E^1$, where $E^1$ is the norm one subgroup of $E$, and  $\mu$ be a character of $E^\times$ such that $\mu|_{F^\times}$ is the local class field theory character associated with $E/F$, we then have an irreducible Weil representation $\omega_{\mu,\psi^{-1},\chi}$ of $\RU(1,1)$. Let $\eta$ be a quasi-character of $E^\times$ and $s\in \BC$, we can consider the induced representation $\Ind_B^H(\eta|~|^{s-1/2})$, where $B$ is the upper triangular Borel subgroup. We assume that the product of the central character of $\pi$, $\omega_{\mu,\psi^{-1},\chi}$ and $\Ind_B^H(\eta|~|^{s-1/2})$ is trivial, and we call our given datum are compatible if it is the case. Suppose $\pi$ is $\psi$-generic. Given $W\in \CW(\pi,\psi) , f_s\in \Ind_B^H(\eta_s), \phi \in \CS(E,\chi)$, where $\CS(E,\chi)$ is the space of the Weil representation of $\omega_{\mu,\psi^{-1},\chi}$, see $\S$1, then the local zeta integral of \cite{GPS} is 
$$\Psi(W,\phi,f_s)=\int_{ZN\setminus H}W(g)(\omega_{\mu,\psi^{-1},\chi}(g)\phi)(1)f_s(g)dg,$$
where $Z$ is the center of $H$.
 
The intertwining operator on the induced representations will give a $\gamma$-factor $\gamma(s,\pi, \omega_{\mu,\psi^{-1},\chi},\eta)$ by a standard uniqueness property, see $\S$2. Note that if $\pi$ is both $\psi$- and $\psi_\kappa$-generic, there are two $\gamma$-factors $\gamma(s,\pi, \omega_{\mu,\psi^{-1}},\eta)$ and $\gamma(s,\pi,\omega_{\mu, \psi_\kappa^{-1},\chi},\eta)$. Our first result shows that they are essentially the same:

\begin{thm}[Proposition \ref{prop214}]
 Let $(\pi,V)$ be an irreducible smooth representation of $H$ which is both $\psi$- and $\psi_\kappa$-generic. Then
   $$\gamma(s,\pi,\omega_{\mu,\psi_\kappa^{-1},\chi},\eta)=\eta(\kappa)|\kappa|_F^{2s-1}\gamma(s,\pi,\omega_{\mu,\psi^{-1},\chi},\eta).$$
\end{thm}
 
The main result of this paper is the following 
\begin{thm}[Theorem \ref{thm39}, Local converse theorem and the stability of $\gamma$-factors for $\RU(1,1)$]\label{thm02}
Suppose that $E/F$ is unramified, or $E/F$ is ramified but the residue characteristic is not $2$. Let $\pi,\pi'$ be two irreducible smooth $\psi$-generic representations of $\RU(1,1)(F)$ with the same central character.
\begin{enumerate}
\item If $\gamma(s,\pi, \omega_{\mu,\psi^{-1},\chi},\eta)=\gamma(s,\pi', \omega_{\mu,\psi^{-1},\chi},\eta)$ for all compatible quasi-characters $\eta$ of $E^\times$ and characters $\mu$ of $E^1$, then $\pi\cong \pi'$. 
\item If $\eta$ is highly ramified, then $\gamma(s,\pi,\omega_{\mu,\psi^{-1},\chi},\eta)=\gamma(s,\pi', \omega_{\mu,\psi^{-1},\chi},\eta)$.
\end{enumerate}
\end{thm}
In \cite{Ba1} and \cite{Ba2}, E.M Baruch proved the local converse theorem for $\RG\Sp_4$ and $\RU(2,1)$ using Howe vectors. Our proof of the above theorem follows Baruch's method closely. In particular, our main tool is also Howe vectors. The main difference is that we need to deal with the Weil representation.

Next, we consider the $\GL_2$ case, which should be viewed as the local theory of the global $\RU(1,1)$ integral at the split places as we mentioned before. In this case, the Weil representation is $\omega_{\psi^{-1},\chi}$ ($\mu$ is trivial in this case) is the induced representation $\Ind_B^H(1\otimes \chi)$. It turns out that the local zeta integral of \cite{GRS} is in fact the local zeta integral for the representation $\pi\times \Ind(1\otimes \chi)\otimes \eta$ which was originally defined by Jacquet in \cite{J}. The analogue of Theorem \ref{thm02} is also true:
\begin{thm}[Theorem \ref{thm511}, Local converse theorem and the stability of $\gamma$-factors for $\GL_2$]\label{thm03}
Let $\pi,\pi'$ be two irreducible smooth $\psi$-generic representations of $\GL_2(F)$ with the same central character.
\begin{enumerate}
\item If $\gamma(s,\pi, \omega_{\psi^{-1},\chi},\eta)=\gamma(s,\pi', \omega_{\psi^{-1},\chi},\eta)$ for all compatible quasi-characters $\eta$ of $F^\times\times F^\times$, then $\pi\cong \pi'$. 
\item If $\eta$ is highly ramified, then $\gamma(s,\pi,\omega_{\mu,\psi^{-1},\chi},\eta)=\gamma(s,\pi', \omega_{\mu,\psi^{-1},\chi},\eta)$.
\end{enumerate}
\end{thm}
In \cite{J}, Jacquet showed the multiplicativity of the $\gamma$-factors: $ \gamma(s,\pi, \omega_{\psi^{-1},\chi},\eta)=\gamma(s,\pi, \chi \eta)\gamma(s,\pi, \eta)$. From this and Part (1) of Theorem \ref{thm03}, we get the classical local converse theorem for $\GL_2$ which was originally proved by Jacquet-Langlands, \cite{JL}.

\section*{Acknowledgement}
I would like to thank my advisor Professor James W.Cogdell for his encouragement, generous support and a lot of invaluable suggestions. Without his suggestions and support, this paper would never exist.

\section*{Notations}
Let $F$ be a local field, and $E/F$ be a quadratic extension. For $x\in E,$ let $x\ra \bar x$ be the unique nontrivial action in $\Gal(E/F)$.  Let $\epsilon_{E/F}:F^\times \ra \wpair{\pm1}$ be the class field theory character of $E/F$. Let $E^1=\wpair{x\in E: x\bar x=1}$. Let $(W,\pair{~,~})$ be the skew Hermitian vector space of dimension 2 with 
$$\pair{w_1,w_2}=w_1J_1{}^t\!\bar w_2,$$
where $w_i\in W$ is viewed as a row vector and $J_1=\begin{pmatrix} &1\\-1&\end{pmatrix}$. Let $H=\RU(1,1)$ be the isometry group of $W$, i.e., 
$$H=\wpair{g\in \GL_E(W)|\pair{w_1g,w_2g}=\pair{w_1,w_2},\forall w_1,w_2\in W}=\wpair{g\in \GL_2(E)|~gs^t\!\bar g=s}.$$
 Let $B$ be the upper triangular subgroup of $H$, then $B=TN$ with
 $$T=\wpair{t(a):=\begin{pmatrix} a&\\ &\bar a^{-1}\end{pmatrix}:a\in E^\times},~ N=\wpair{n(b):=\begin{pmatrix} 1& b\\&1\end{pmatrix},b\in F}.$$
Let $Z=\wpair{t(a):a\in E^1}$ be the center of $H$ and $R=ZN$. We denote $\bar N$ (resp. $\bar B$) the lower triangular unipotent subgroup (resp. lower triangular subgroup) of $\RU(1,1)$. For $x\in F$, we denote
$$\bar n(x)=\begin{pmatrix} 1& \\ x&1\end{pmatrix}\in \bar N.$$

 \section{Weil representations of $\RU(1,1)$}
\subsection{An exact sequence} Let $\psi$ be a nontrivial additive character of $F$ which is viewed as a character of $N$ by the natural isomorphism $N\cong F$. Fix an element $\kappa\in F^\times-\Nm(E^\times)$. Let $(\pi,V)$ be an infinite dimensional representation of $B$. As a representation of $N$, $(\pi,V)$ is smooth. As a smooth module over $N$, we have $\CS(N).V=V$, see 2.5 of \cite{BZ1}, where $\CS(N)$ is the space of Bruhat-Schwartz functions on $N$. Here we also view $\CS(N)$ the Hecke algebra on $N$. By Fourier inversion formula, we have an isomorphism $\CS(N)\cong \CS(\hat N)$, where $\hat N$ is the dual group of $N$. Under this isomorphism, we have $\CS(\hat N).V=V$. Thus there is a sheaf $\CV$ on $\hat N$ such that $\CV_c=V$, see 1.14 of \cite{BZ1}.  Consider the action of $B$ on $\hat N$ as follows: for $\psi\in \hat N$, and $b\in B$, define
$$b.\psi(n)=\psi(b^{-1}nb).$$
This gives an action of $B$ on $\CV$. The action of $B$ on $\hat N$ has 3 orbits, $\wpair{0}$, $\hat N_1=\wpair{b.\psi: b\in B}$ and $\hat N_2=\wpair{b.\psi_\kappa: b\in B}$. Note that $\hat N_1\cup \hat N_2$ is open in $\hat N$. The stabilizer of $\psi$ is $R$. By 2.23 and 5.10 of \cite{BZ1}, we have 
$$\CV(\hat N_1)=\ind_R^B(V_{N,\psi}), \textrm{ and } \CV(\hat N_2)=\ind_R^B(V_{N,\psi_{\kappa}}),$$
where $\ind_R^B$ denote the non-normalized compact induction. Now the exact sequence $$0\ra \CV(\hat N_1)\oplus \CV(\hat N_2)\ra \CV(\hat N)\ra \CV(\wpair{0})\ra 0,$$
see 1.16 of \cite{BZ1}, can be identified with
\begin{equation}{\label{eq11}}
0\ra \ind_R^B(V_{N,\psi})\oplus \ind_R^B(V_{N,\psi_\kappa})\ra V\ra V_N\ra 0.
\end{equation}
We call $(\pi,V)$ is $\psi$ (resp. $\psi_\kappa$)-generic if $V_{N,\psi}$ (resp. $V_{N,\psi_\kappa}$) is nonzero. If $V_{N,\psi}\ne 0$, then a nonzero element 
$$\lambda\in \Hom_N(V,\psi)=\Hom(V_{N,\psi},\BC)$$
is called a $\psi$-th Whittaker functional of $(\pi,V)$.
\begin{prop}{\label {prop11}}
Let $(\pi,V)$ be an irreducible smooth admissible representation of $H=\RU(1,1)(F)$. Then we have
\begin{enumerate}
\item if $V_{N,\psi}\ne 0$, then $V_{N,\psi_{a\bar a}}\ne 0$ for any $a\in E^\times ;$
\item if $V_{N,\psi}=V_{N,\psi_\kappa}=0$, then $V$ has finite dimension$;$
\item we have $\dim V_{N,\psi}\le 1$ and $\dim V_{N,\psi_\kappa}\le 1;$  if $\dim V_{N,\psi}=1$, then $\ind_R^B(V_{N,\psi})$ is an irreducible representation of $B$. Moreover, $\ind_R^B(V_{N,\psi})$ is not equivalent to $\ind_R^B(V_{N,\psi_\kappa})$ if at least one of them is nonzero.
\end{enumerate}
\end{prop}
\begin{proof}
(1) If $\lambda\in \Hom_N(V,\psi)$, it is easy to check that $\lambda\circ \pi(t(a))\in \Hom_N(V,\psi_{a\bar a})$.

(2) The assertion follows from the exact sequence (\ref{eq11}).

(3) The first part is the uniqueness of the Whittaker model and is well-known. Suppose that $\dim{V_{N,\psi}}=1$. The proof of the fact that $\ind_R^B(V_{N,\psi})$ is irreducible is similar to the proof of the corresponding statement in the $\GL_n$ case, see 5.13 of \cite{BZ1}. We give a sketch here. Let $V_1'$ be a nonzero $B$ submodule of $V':=\ind_R^B(V_{N,\psi})$. It is not hard to check that $(V'_1)_N=0$ and $(V'_1)_{N,\psi_\kappa}=0$ by Jacquet-Langlands Lemma, 2.33 of \cite{BZ1}. Thus by the exact sequence (\ref{eq11}), we have $V_1'=\ind_R^B((V_1')_{N,\psi})$. In particular, we have $(V_1')_{N,\psi}\ne 0$. On the other hand, if we take $V_1'=V_1$ in the above argument, we obtained $V'=\ind_R^B(V'_{N,\psi})$. Since $\ind_R^B(V_{N,\psi})=V'=\ind_R^B(V'_{N,\psi})$ and $V_{N,\psi}$ has dimension one, we conclude $\dim V'_{N,\psi}=1$.  By the exactness of Jacquet functors we have $0\ne (V_1')_{N,\psi}\incl (V')_{N,\psi}$. Thus $(V_1')_{N,\psi}= (V')_{N,\psi}$, and $V_1'=\ind_R^B((V_1')_{N,\psi} )=\ind_R^B((V')_{N,\psi})=V'$. This shows that the only nonzero $B$-submodule of  $V'$ is $V'$ itself, and thus $V'$ is an irreducible $B$-module. 

Suppose that both of $V_{N,\psi}$ and $V_{N,\psi_\kappa}$ are nonzero. As we mentioned before, one can check that $(\ind_R^B(V_{N,\psi}))_{N,\psi_\kappa}=0$ but $(\ind_R^B(V_{N,\psi_\kappa}))_{N,\psi_\kappa}\ne 0$, thus $ \ind_R^B(V_{N,\psi})$ is not equivalent to $\ind_R^B(V_{N,\psi_\kappa})$. This completes the proof.
\end{proof}

If $(\pi,V)$ is a representation of $\RU(1,1)$ such that exactly one of $V_{N,\psi}$ and $V_{N,\psi_\kappa}$ is nonzero, we call $\pi$ is \textbf{exceptional}.

We will show that the Weil representations provide examples of exceptional representations. 
\subsection{Weil representations}
Let $V$ be a Hermitian space of dimension 1. We have the dual pair $\RU(V)\times \RU(W)$, recall that $\RU(W)=\RU(1,1)$. For a character $\mu$ of $E^\times$ with $\mu|_{F^\times}=\epsilon_{E/F}$, we have a splitting $s_\mu: \RU(V)\times \RU(W)\ra Mp(V\otimes W)$, see \cite{HKS} for example. For a nontrivial additive character $\psi$ of $F$, we then have a Weil representation $\omega_{V,\psi,\mu}$ of $\RU(V)\times \RU(W)$ on $\CS(V)=\CS(E)$. We have the following formulas
\begin{align}
\omega(g,1)\phi(x)&=\phi(g^{-1}x), g\in U(V),\\
\omega(1,t(a))\phi(x)&=\mu(a)|a|^{1/2}\phi(xa)\\
\omega(1,n(b))\phi(x)&=\psi((bx,x)_V)\phi(x)\\
\omega(1,w)\phi(x)&=\gamma_\psi\int_{V}\psi(-\tr_{E/F}(x,y)_V)\phi(y)dy, \label{eq15}
\end{align}
where in the last formula, $$w=\begin{pmatrix} &1\\ -1&\end{pmatrix}$$
is the unique nontrivial Weyl group element of $\RU(1,1)$, $\gamma_\psi$ is the Weil index, and $dy$ is the measure on $V$ which is self-dual for the Fourier transform defined by Eq.(\ref{eq15}).

\begin{prop}\begin{enumerate}
\item For $a\in \Nm(E^\times)$, we have $\omega_{V,\psi_a, \mu}\cong \omega_{V,\psi,\mu}$.
\item As a representation of $\RU(W)=\RU(1,1)$, the restriction of the representation $\omega_{V,\mu,\psi}$ to $Z\cong E^1$ is fully reducible. For a character $\chi$ of $Z\cong E^1$, let $ \omega_{V,\mu,\psi,\chi}$ be the representation of $\RU(1,1)$ on the $\chi$-eigenspace $\CS(E,\chi)$ in $(\omega_{V,\mu,\psi}, \CS(E))$, then $\omega_{V,\mu,\psi,\chi}$ is irreducible. Moreover, $\omega_{V,\mu,\psi,\chi}$ is supercuspidal if $\chi$ is nontrivial, and $\omega_{V,\mu,\psi,1}$ is a direct summand of $\Ind_B^H(\mu)$, where $1$ means the trivial character of $E^1$, and $\Ind_B^H(\mu)$ is the normalized induction.
\end{enumerate}
\end{prop}
This is a special case of the general theorem of Kudla in the symplectic case and then Moeglin-Vigneras-Waldspurger. See p.69 of \cite{MVW}, or Proposition 2.5.1 of \cite{GR}.

For $a\in F^\times$, let $V_a$ be the 1-dimensional Hermitian space with $V_a=E,$ and the Hermitian structure $Q_a(x,y)=ax\bar y$. We will write $\omega_{a,\mu,\psi}$ (resp. $\omega_{a,\mu,\psi,\chi}$) for $\omega_{V_a,\mu,\psi}$ (resp. $\omega_{V_a, \mu,\psi,\chi}$) and $\omega_{\mu,\psi}$ (resp. $\omega_{\mu,\psi,\chi}$) for $\omega_{1,\mu,\psi}$ (resp. $\omega_{1,\mu,\psi,\chi}$).

\begin{prop} We fix a $\kappa\in F^\times-\Nm(E^\times)$.
\begin{enumerate}
\item For $a\in F^\times$, we have $\omega_{a,\mu,\psi,\chi}\cong \omega_{\mu,\psi_a,\chi}$.
\item The representation $\omega_{\mu,\psi,\chi}$ is $\psi$-generic but not $\psi_\kappa$-generic. Similarly, $\omega_{\mu,\psi_\kappa,\chi}$ is $\psi_\kappa$-generic but not $\psi$-generic. Thus $\omega_{\mu,\psi_a,\chi}$ is an exceptional representation of $\RU(1,1)$ for all $a\in F^\times$.
\end{enumerate}
\end{prop}
\begin{proof}
(1) In fact, the identity map $\CS(E,\chi)\ra \CS(E,\chi)$ defines an isomorphism $\omega_{a,\mu,\psi,\chi}\cong \omega_{\mu,\psi_a,\chi}$, see Corollary 6.1 of \cite{K}, page 40.

(2) This is in fact proved in \cite{KS}. We also include a proof here. Denote $(\pi,V)$ the representation $(\omega_{\mu,\chi,\psi},\CS(E,\chi))$ temporarily. We claim that $V=V(N,\psi_\kappa)$, i.e., $V_{N,\psi_\kappa}=0$. Given $\phi\in V=\CS(E,\chi)$, by Jacquet-Langlands' Lemma, 2.33 of \cite{BZ1}, we have $\phi\in V(N,\psi_\kappa)$ if and only if there is an open compact subgroup $N'\subset N$ such that
$$\int_{N'}\psi_\kappa^{-1}(n)\pi(n)\phi dn=0,$$
or for all $x\in \Supp(\phi)$, 
\begin{equation}{\label{eq16}}\int_{N'}\psi(n(x\bar x-\kappa))dn\equiv 0.\end{equation}
Since $\Supp(\phi)$ is compact and $\kappa \notin \Nm(E^\times)$, we can find an open compact subgroup $N'$ such that $\psi(n(x\bar x-\kappa))$ is a nontrivial character on $N'$ for all $x\in \Supp(\phi)$. Then for this $N'$, (\ref{eq16}) holds, i.e., $\phi\in V(N,\psi_\kappa)$. Thus $V_{N,\psi_\kappa}=0$ and $\omega_{\mu,\psi,\chi}$ is not $\psi_\kappa$-generic. It is clear that $\omega_{\mu,\chi,\psi}$ is $\psi$-generic. 
\end{proof}

\begin{lem}{\label {lem14}}
Define a linear functional $\lambda_{\mu,\psi,\chi}:\CS(E,\chi)\ra \BC$ by 
$$\lambda_{\mu,\psi,\chi}(\phi)=\phi(1).$$
Then $\lambda_{\mu,\psi,\chi}$ is a nonzero Whittaker functional of $\omega_{\mu,\psi,\chi}$.
\end{lem}
\begin{proof}
For $n(x)\in N$, we have 
$$\lambda_{\mu,\psi,\chi}(\omega(n(x))\phi)=\omega(n(x))\phi(1)=\psi(x)\phi(1).$$
Thus $\lambda_{\mu,\psi,\chi}$ is a Whittaker functional. It is clear that $\lambda_{\mu,\psi,\chi}$ is nonzero.
\end{proof}

If $\chi=1$ is the trivial character of $E^1$, we have $\omega_{\mu,\psi,1}\subset \Ind_B^H(\mu)$. Kudla and Sweet analyzed the embedding in \cite{KS}. In the $\RU(1,1)$ case, the main result of \cite{KS} is the following

\begin{thm}[Kudla-Sweet]{\label{thm15}} Let $\mu$ be a character of $E^\times$, and $s\in \BC$. We normalize $(s,\mu)$ as in page $255$ of \cite{KS}. Let $I(s,\mu)=\Ind_B^H(\mu|~|^s)$. Then
\begin{enumerate}
\item If $\mu|_{F^\times}\ne 1,$ and $\mu|_{F^\times}\ne \epsilon_{E/F}$. Then $I(s,\mu)$ is irreducible.
\item If $\mu|_{F^\times}=1$, then $I(s,\mu)$ is irreducible except when $s=\pm \frac{1}{2}$. If $s=\pm \frac{1}{2}$, $I(s,\mu)$ has length 2, and one irreducible component is has dimension 1.
\item If $\mu|_{F^\times}=\epsilon_{E/F}$, then $I(s,\mu)$ is irreducible except when $s=0$. If $s=0$, we have
$$I(0,\mu)=\omega_{\mu,\psi,1}\oplus \omega_{\mu,\psi_\kappa,1}.$$
\end{enumerate}
\end{thm}

\section{Local zeta integrals for $\RU(1,1)$}
\subsection{The Local Zeta Integral}
Let $\psi$ be a fixed nontrivial additive character on $F$.  Let $(\pi,V)$ be a $\psi$-generic irreducible smooth representation of $H=\RU(1,1)(F)$ and let $\omega_\pi$ be the central character of $\pi$. Let $\mu$ be a character of $E^\times$ such that $\mu|_{F^\times}=\epsilon_{E/F}$. For a character $\chi$ of $E^1$, we can consider the Weil representation $\omega_{ \mu,\psi^{-1},\chi}$ of $H=U(1,1)$. For a character $\eta$ of $E^\times$, $s\in \BC$, let $\eta|~|^{s-1/2}$ be the character of $B$ defined by 
$$\eta|~|^{s-1/2}(nt(a))=\eta(a)|a|_E^{s-1/2}, a\in E^\times, n\in N.$$
Let $\Ind_B^H(\eta|~|^{s-1/2})$ be the normalized induced representation of $H$, which consists smooth functions on $H$ such that
$$f(bh)=\eta|~|^s(b)f(h),b\in B, h\in H.$$
We assume that our datum $\pi, \mu, \chi,$ and $ \eta$ satisfy the condition
\begin{equation}{\label{eq21}}
\omega_\pi\cdot \mu\chi\cdot \eta|_{E^1}=1.
\end{equation}

 For $W\in \CW(\pi,\psi)$, $\phi\in \CS(E,\chi)$, $f_s\in \Ind_B^H(\eta|~|^{s-1/2})$, we consider the local zeta integral
 $$\Psi(W,\phi,f_s, \omega_{\mu,\psi^{-1},\chi})=\int_{R\setminus H}W(h)(\omega_{\mu,\psi^{-1},\chi}(h)\phi)(1)f_s(h)dh.$$
There is a projection $\CS(E)\ra \CS(E,\chi)$ defined by $\phi\mapsto \phi_\chi$, where 
$$\phi_\chi(a)=\int_{E^1}\chi^{-1}(u)\phi(ua)du.$$
If we start from an element $\phi\in \CS(E)$, and consider $\Psi'(W,\phi,f_s):=\Psi(W,\phi_\chi, f_s)$, then
$$\Psi'(W,\phi,f_s, \omega_{\mu,\psi^{-1},\chi})=\int_{N\setminus H}W(h)(\omega_{\mu,\psi^{-1}}(h)\phi)(1)f_s(h)dh.$$
\textbf{Remark:} (1) The term $\omega_{\mu,\psi^{-1},\chi}(h)\phi(1)$ is the Whittaker function of the Weil representation associated with the vector $\phi$. \\
(2) The integral $\Psi'(W,\phi, f_s, \omega_{\mu,\psi^{-1},\chi})$ is the local zeta integral in $\S$7 of \cite{GRS}.\\
(3) In the notation $\Psi(W,\phi,f_s, \omega_{\mu,\psi^{-1},\chi}) $, we add $\omega_{\mu,\psi^{-1},\chi}$ to emphasize that the integral is defined in a way depending on the Weil representation $\omega_{\mu,\psi^{-1},\chi}$. If the Weil representation $\omega_{\mu,\psi^{-1},\chi}$ is clear from the context, we will write $\Psi(W,\phi,f_s, \omega_{\mu,\psi^{-1},\chi})$ as $\Psi(W,\phi, f_s)$ for simplicity. 

\begin{lem}{\label{lem21}}
The integral $\Psi(W,\phi,f_s)$ is absolutely convergent for $\Re(s)>>0$, and defines a rational function of $q_E^{-s}$. Moreover, each $\Psi(W,\phi, f_s)$ can be written with a common denominator determined by $\pi, \omega_{\mu,\psi^{-1},\chi}$ and $\eta$.
\end{lem}
\begin{proof}
This follows from a gauge estimate of the Whittaker function $W$. The proof is similar to the proof of the well-known $\GL_n$ case which can be found in \cite{JPSS} or \cite{C} for example. We omit the details. 
\end{proof}

\subsection{The normalized local zeta integral and the local $L$-factor}
We follow \cite{Ba1} to give a parametrization of the induced representation of $\Ind_B^H(\eta|~|^{s-1/2})$ using the Bruhat-Schwartz function space $\CS(F^2)$. Let $\psi'$ be another fixed additive character of $F$ which can be the same or different with $\psi$. For $\Phi\in \CS(F^2)$, we define the Fourier transform with respect to $\psi'$ by:
$$\hat \Phi(x,y)=\int \Phi(u,v)\psi'(yu-xv)dudv.$$
Let $g\in \GL(2,F).$ Set $(g\Phi)(x,y)=\Phi((x,y)g)$. Then 
$$(g\Phi)^{\widehat~}=|\det g|_F^{-1}g'\hat\Phi,$$
where $g'=\diag(\det(g)^{-1},\det(g)^{-1})g$.
For $s\in \BC, g\in \GL(2,F),\Phi\in \CS(F^2)$ and a character $\eta$ of $E^\times$, we consider 
$$z(s,g,\Phi,\eta)=\int_{F^\times}(g\Phi)(0,r)\eta(r)|r|_E^s dr.$$
The above integral is absolutely convergent for $\Re(s)>>0$ and defines a meromorphic function on $\BC$. 

Note that $\SL_2(F)$ is a subgroup of $H$. The determinant map $\det: H\ra E^1$ induces an exact sequence
$$1\ra \SL_2(F)\ra H\ra E^1\ra 1,$$
i.e., $\SL_2=SU(1,1)$.

By Hilbert's Theorem 90, for $h\in H$, we can find $a\in E^\times$ such that $h=t(a)g$ for some $g\in \SL_2(F)$. The decomposition $h=t(a)g$ is not unique. We define
\begin{equation}
f(s,h,\Phi,\eta)=\eta(a)|a|^sz(s,g\Phi, \eta).
\end{equation}
By Lemma 2.5 of \cite{Ba1}, the definition of $f$ is independent of choice of the decomposition of $h$. It is clear that $f\in \Ind_B^H(\eta|~|^{s-1/2})$. By Lemma 4.2 of \cite{Ba1}, there exists $s_0\in \BR$ such that for every $s$ with $\Re(s)>s_0$ and $f\in \Ind_B^H(\eta|~|^{s-1/2})$, there exists $\Phi\in \CS(F^2)$ such that 
$$f(s,h,\Phi,\eta)=f(h).$$

We assume $\pi$, $\mu,\chi$ and $\eta$ satisfy the condition (\ref{eq21}). We define
\begin{equation}
\Psi(s, W,\phi,\Phi,\eta, \omega_{\mu,\psi^{-1}, \chi})=\int_{R\setminus H}W(h)(\omega_{\mu,\psi^{-1},\chi}(h)\phi)(1)f(s,h,\Phi,\eta)dh.
\end{equation}
\textbf{Remark:} Again, if the Weil representation $\omega_{\mu,\psi^{-1},\chi}$ is clear from the context, we will omit that from the notation and just write $\Psi(s,W,\phi, \Phi, \eta)$.

By Lemma \ref{lem21}, for $\Re(s)>>0$, the local zeta integral $\Psi(s,W,\phi,\Phi,\eta)$ is absolutely convergent and defines a rational function of $q_E^{-s}$.

Let $I(s,\pi,\omega_{\mu,\psi^{-1},\chi},\eta)$ be the subspace of $\BC(q_E^{-s})$ spanned by $\Psi(s,W,\phi,\Phi,\eta, \omega_{\mu,\psi^{-1}_{a\bar a},\chi})$ for $W\in \CW(\pi,\psi_{a\bar a}),\phi\in \CS(E,\chi), $, $a\in E^\times$, $\Phi\in \CS(F^2)$.  Note that if $\pi$ is $\psi$-generic, then it is also $\psi_{a\bar a}$-generic for all $a\in E^\times$.
\begin{prop}{\label{prop22}}
The space $I(s,\pi,\omega_{\mu,\psi^{-1},\chi},\eta)$ is a $\BC[q_E^s,q_E^{-s}]$-fractional ideal. There is a unique generator $L(s,\pi,\omega_{\mu,\psi^{-1},\chi},\eta)$ of $I(s,\pi,\omega_{\mu,\psi^{-1},\chi},\eta)$ of the form $P(q_E^{-s})^{-1}$ for $P(X)\in \BC[X]$ and $P(1)=1$. 
\end{prop}

\begin{proof}
For $\phi\in \CS(E,\chi), a\in E^\times$, define $\phi_a(x)=\phi(ax)$. One can check that $\phi\mapsto \phi_a$ defines an isomorphism $\omega_{\mu,\psi^{-1},\chi}\cong \omega_{\mu,\psi^{-1}_{a\bar a},\chi}$. In particular, we have $(\omega_{\mu,\psi^{-1},\chi}(h)\phi)_a(x)=(\omega_{\mu,\psi^{-1}_{a\bar a},\chi}(h)\phi_a)(x)$. 

Let $W\in \CW(\pi,\psi)$ and $a\in E^\times$ , we define $W_a(h)=W(t(a)h)$. Then $W_a\in \CW(\pi,\psi_{a\bar a})$. Take $\phi\in \CS(E,\chi)$, we have
 \begin{align*}&\quad \Psi(s,W_a,\phi_a,\Phi,\eta, \omega_{\mu,\psi^{-1}_{a\bar a},\chi})\\
 &=\int_{R\setminus H}W(t(a)h)\omega_{\mu,\psi^{-1}_{a\bar a},\chi}(h)\phi_a(1)f(s,h,\Phi,\eta)dh\\
 &=\int_{R\setminus H}W(t(a)h)\omega_{\mu,\psi^{-1},\chi}(h)\phi(a)f(s,h,\Phi,\eta)dh\\
 &=\mu(a)^{-1}|a|^{-1/2}\int_{R\setminus H}W(t(a)h)\omega_{\mu,\psi^{-1},\chi}(t(a)h)\phi(1)f(s,h,\Phi,\eta)dh\\
 &=\mu(a)^{-1}|a|^{-1/2}\int_{R\setminus H}W(h)\omega_{\mu,\psi^{-1},\chi}(h)f(s,t(a)^{-1}h,\Phi,\eta)dh\\
 &=\mu(a)^{-1}\eta(a)^{-1}|a|_E^{-s-1/2}\Psi(s,W,\phi,\Phi,\eta, \omega_{\mu,\psi^{-1}, \chi}).
 \end{align*}
 If we take $a$ to be a prime element of $E$, then $|a|_E=q_E^{-1}$. Thus $I(s,\pi, \omega_{\mu,\psi^{-1},\chi},\eta)$ is closed under multiplication by $q_E^{s}$. Since $\Psi(s,W,\theta, \Phi, \eta)$ has bounded denominators,  it is a $\BC[q_E^{-s},q_E^s]$-factional ideal. Since $\BC[q_E^s,q_E^{-s}]$ is a principal ideal domain, it has a generator. To show the generator has the given form, it suffices to show that we can find $W,\phi,\Phi$ such that 
 $$\Psi(s,W,\phi,\Phi, \eta)=1.$$
 This will be proved in next section using Howe vectors, see the proof of Theorem \ref{thm39} and Remark \ref{rem310}.
\end{proof}

The next Lemma says that, when $E/F$ is unramified, to obtain the whole ideal $I(s,\pi, \omega_{\mu,\psi^{-1},\chi},\eta)$, we do not have to vary $a\in E^\times$.
\begin{lem}
Suppose $E/F$ is unramified. Let $I'(s,\pi,\omega_{\mu,\psi^{-1},\chi},\eta)$ be the space spanned by $\Psi(s,W,\phi,\Phi,\eta)$ for $W\in \CW(\pi,\psi),\phi\in \CS(E,\chi),$ and $\Phi\in \CS(F^2)$. Then $I'(s,\pi,\omega_{\mu,\psi^{-1},\chi},\eta)$ is a fractional ideal and $I'(s,\pi,\omega_{\mu,\psi^{-1},\chi},\eta)=I(s,\pi,\omega_{\mu,\psi^{-1},\chi},\eta)$.
\end{lem}
\begin{proof}
Let $a\in F^\times$, consider 
$$\Phi_a(x)=\Phi(ax), x\in F^2.$$
Notice that $f(s,h,\Phi_a,\eta)=\eta^{-1}(a)|a|_E^{-s}f(s,h,\Phi,\eta)$. Thus
$$\Psi(s,W,\phi,\Phi_a,\eta)=\eta^{-1}(a)|a|_E^{-s}\Psi(s,W,\phi,\Phi,\eta).$$
Since $E/F$ is unramified, we can take $a$ to be a prime element of $E$ and $a\in F$, then $|a|_E=q_E^{-1}$. Thus $I'$ is a fractional ideal. From the calculation in the proof of Proposition \ref{prop22}, we have 
$$\Psi(s,W_a,\phi_a,\Phi,\eta, \omega_{\mu,\psi^{-1}_{a\bar a},\chi})=\mu^{-1}(a)|a|^{-1/2}\Psi(s,W,\phi,\Phi_a,\eta, \omega_{\mu,\psi^{-1},\chi}).$$
Thus $I=I'$.
\end{proof}

\subsection{The Local Functional Equation}
Consider the intertwining operator
$$M(s): \Ind_B^H(\eta|~|^{s-1/2})\ra \Ind_B^H(\eta^*|~|^{1/2-s})$$
$$(M(s)f_s)(h)=\int_N f_s(wnh)dn,$$
with $w=\begin{pmatrix}&1\\ -1& \end{pmatrix}$. It is well-known that this operator is well-defined for $\Re(s)>>0$ and can be meromorphically continued to all $s\in \BC$. 

By Lemma 14.7.1 of \cite{J}, there is a meromorphic function $c_0(s)$ such that
$$(M(s)f(s,\cdot,\Phi,\eta))(h)=c_0(s)f(1-s,h,\hat \Phi, \eta^*),$$
where $\hat \Phi$ is the Fourier transform defined by $\psi'$ and we omit $\psi'$ from the notation.

\begin{cor}{\label{cor24}}
For $\Re(s)>0$, there are unique $H$-invariant trilinear forms $\beta_s$ and $\beta_s'$ on $\CW(\pi,\psi)\times \omega_{\mu,\psi^{-1},\chi}\times \Ind_B^H(\eta|~|^{s-1/2})$ such that if $f$ is the function defined by $f(h)=f(s,h,\Phi,\eta)$, then
$$\beta_s(W,\phi,f)=\Psi(s,W,\phi,\Phi,\eta),$$
and 
$$\beta_s'(W,\phi,f)=\Psi(1-s,W,\phi,\hat \Phi,\eta^*).$$
\end{cor}

Let $\CT_s$ be the space of $H$-invariant trilinear forms on $\CW(\pi,\psi)\times \CS(E,\chi)\times \Ind_B^H(\eta|~|^{s-1/2})$.
\begin{prop}{\label{prop25}}
Except for a finite number of values of $q_E^{-s}$, we have 
$$\dim \CT_s= 1.$$
\end{prop}
This proposition can be viewed as a special case of the uniqueness of the Fourier-Jacobi model, which was proved by Binyong Sun in \cite{Su}. We also include a proof here based on the method Jacquet used in \cite{J}.

To prove Proposition \ref{prop25}, we need the following
\begin{lem}{\label{lem26}}
Let $(\pi,V_\pi)$ be an irreducible representation of $H$, $(\sigma,V_\sigma)$ be an irreducible exceptional representation of $H$, $\nu$ be a quasi-character of $E^\times$. Then except a finite number of values of $q_E^s$, there is at most one bilinear form $\CB_s:V_\pi\times V_\sigma\ra \BC$ such that 
\begin{equation}{\label{eq22}}
\CB_s(\pi(b)v,\sigma(b)v')=\nu_s(b)\CB_s(v,v'), \quad \forall b\in B,
\end{equation}
where for $b=t(a)n(x)\in B$, $\nu_s$ is defined by $\nu_s(b)=\nu(a)|a|_E^s$.
\end{lem}
\begin{proof}
We suppose that $\sigma$ is $\psi$-generic. Let $\CB_s $ and $\CB_s'$ be two nonzero bilinear forms  $V_\pi\times V_\sigma\ra \BC$ satisfying (\ref{eq22}). By Proposition \ref{prop11}, $V_\pi(N)=\ind_R^B((V_\pi)_{N,\psi})\oplus \ind_R^B((V_\pi)_{N,\psi_\kappa})$ is a sum of two non-equivalent irreducible $B$-modules (one of them might be zero), $V_\sigma(N)=\ind_R^B((V_\sigma)_{N,\psi})$ is an irreducible $B$-module. Thus there is a constant $c(s)$ such that $\CB_s'|_{V_\pi(N)\times V_\sigma(N)}=c(s)\CB_s|_{V_\pi(N)\times V_\sigma(N)}$. 

For $v\in V_\pi, n\in N, v'\in V_\sigma(N)$, since $\nu_s(n)=1$, we have 
$$\CB_s(v,v'-\sigma(n^{-1})v')=\CB_s(v,v')-\CB_s(\pi(n)v,v')=\CB_s(v-\pi(n)v,v').$$
Similarly $$\CB_s'(v,v'-\sigma(n^{-1})v')=\CB_s'(v-\pi(n)v,v').$$ Since $v-\pi(n)v\in V_\pi(N),v'\in V_\sigma(N)$, we have
\begin{equation}{\label{eq23}}\CB_s'(v,v'-\sigma(n^{-1})v')=\CB_s'(v-\pi(n)v,v')=c(s)\CB_s(v-\pi(n)v,v')=c(s)\CB_s(v,v'-\sigma(n^{-1})v').\end{equation}
Since $V_\sigma^N=\wpair{0}$, we can find $v'\in V_\sigma(N), n\in N$ such that $v'-\sigma(n^{-1}v')\ne 0$. Since $V_\sigma(N)$ is irreducible, we get 
\begin{equation}{\label{eq24}}\CB_s'|_{V_\pi\times V_\sigma(N)}=c(s)\CB_s|_{V_\pi\times V_\sigma(N)}\end{equation} by (\ref{eq23}).

For $v\in V_\pi,n\in N,v'\in V_\sigma$, similar to (\ref{eq23}), we have
$$\CB_s'(v-\pi(n)v,v')=\CB'_s(v,v'-\sigma(n^{-1})v')=c(s)\CB_s(v,v'-\sigma(n^{-1})v')=c(s)\CB_s(v-\pi(n)v,v'),$$
since $v'-\sigma(n^{-1})v'\in V_\sigma(N)$ and by (\ref{eq24}). Since $V_\pi(N)$ is spanned by $ v-\pi(n)v$ for $v\in V_\pi, n\in N$, we get
\begin{equation}{\label{eq25}}
\CB_s'|_{V_\pi(N)\times V_\sigma}=c(s)\CB_s|_{V_\pi(N)\times V_\sigma}.
\end{equation}
By (\ref{eq24}), (\ref{eq25}), the bilinear form $\CC_s=\CB_s'-c(s)\CB_s$ is zero on $V_\pi\times V_\sigma(N)$ and $V_\pi(N)\times V_\sigma$, thus defines a linear form $(V_\pi)_N\otimes (V_\sigma)_N\ra \BC$ such that
\begin{equation}{\label{eq28}}\CC_s(\pi(b)v\otimes \sigma(b)v')=\nu_s(b)\CC_s(v\otimes v').\end{equation}
Denote $W$ the $B$-space $(V_\pi)_N\otimes (V_\sigma)_N$ temporally. Then (\ref{eq28}) implies that if $\CC_S$ is nonzero, then $W$ has a 1-dimensional quotient $(\nu_s,\BC)$. On the other hand, the space $W$ has finite dimension, in fact has at most dimension 2. Thus it has finite number 1-dimensional quotient $\chi_i$. Thus if $\nu_s\ne \chi_i$, we have $\CC_s=0$, i.e., except for a finite number of values of $q_E^s$, $\CC_s$ is identically zero, i.e., 
$$ \CB_s'\equiv c(s)\CB_s.$$
\end{proof}
\begin{proof}[Proof of Proposition $\ref{prop25}$]
In the following proof, we use $\Ind_B^H(\eta|~|^s)$ to denote the non-normalized induction so that the Frobenius reciprocity has a simpler form. By Frobenius reciprocity, we have
\begin{align*}
\CT_s
=&\Hom_H(\pi\otimes \omega_{\mu,\psi^{-1},\chi}, \Ind_B^H(\eta^{-1}|~|^{1-s}))\\
=&\Hom_B(\pi\otimes \omega_{\mu,\psi^{-1},\chi}, \eta^{-1}|~|^{1-s})\\
=&\Hom_B(\pi\otimes \eta|~|^{s-1}|_B, \widetilde\omega_{\mu,\psi^{-1},\chi}|_B).
\end{align*}
Note that $\omega_{\mu,\psi^{-1},\chi}$ is exceptional, thus by Lemma \ref{lem26}, we have 
$$\dim \CT_s\le 1.$$
As noted in the proof of Proposition \ref{prop22}, in next section, we will show that there exists $W\in \CW(\pi,\psi^{-1}),\phi\in \CS(E,\chi)$ and $\Phi\in \CS(F^2)$ such that
$$\Psi(s,W,\phi,\Phi,\eta)=1,$$
except a finite number value of $q_E^s$ after meromorphic continuation.
Thus $\dim \CT_s=1$.
\end{proof}
As a corollary of Corollary \ref{cor24} and Proposition \ref{prop25}, we have 
\begin{cor}
There exists a meromorphic function $\gamma(s,\pi,\omega_{\mu,\psi^{-1},\chi},\eta,\psi')$ such that
$$ \Psi(1-s,W,\theta,\hat \Phi,\eta^*)=\gamma(s,\pi,\omega_{\mu,\psi^{-1},\chi},\eta,\psi')\Psi(s,W,\phi,\Phi,\eta)$$
\end{cor}
Note that in the notation, $\psi$ is used to define the Weil representation $\omega_{\mu,\psi^{-1},\chi}$ and $\psi'$ is used to define the Fourier transform $\hat \Phi$.

We also define the $\epsilon$-factor
$$\epsilon(s,\pi,\omega_{\mu,\psi^{-1},\chi},\eta,\psi')=\gamma(s,\pi,\omega_{\mu,\psi^{-1},\chi},\eta,\psi')\frac{L(s,\pi,\omega_{\mu,\psi^{-1},\chi},\eta)}{L(1-s,\pi,\omega_{\mu,\psi^{-1},\chi},\eta^*)}.$$

\subsection{Dependence on $\psi$}
In the definition, the $\gamma$-factor dependa on choices of $\psi$ and $\psi'$. In this section, we shall consider the (in)dependence on $\psi$ and $\psi'$. 

\begin{lem}{\label{lem28}} For $a\in E^\times,b\in F^\times$, we have
 \begin{align*} L(s,\pi,\omega_{\mu,\psi_{a\bar a}^{-1},\chi}, \eta)&=L(s,\pi,\omega_{\mu,\psi^{-1},\chi},\eta), \\
\gamma(s,\pi,\omega_{\mu,\psi^{-1}_{a\bar a},\chi},\eta,\psi')&=\eta(a\bar a)|a\bar a|_F^{2s-1}\gamma(s,\pi,\omega_{\mu,\psi^{-1},\chi},\eta,\psi'),\end{align*}
and 
$$\gamma(s,\pi,\omega_{\mu,\psi^{-1},\chi},\eta,\psi'_b)=\eta(b)|b|_E^{s-1}\gamma(s,\pi,\omega_{\mu,\psi^{-1},\chi},\eta,\psi').$$
\end{lem}
\begin{proof} 
 By definition, it is clear that $I(s,\pi,\omega_{\mu,\psi^{-1}_{a\bar a},\chi},\eta)=I(s,\pi,\omega_{\mu,\psi^{-1},\chi},\eta)$, and thus
$$ L(s,\pi,\omega_{\mu,\psi^{-1}_{a\bar a},\chi}, \eta)=L(s,\pi,\omega_{\mu,\psi^{-1},\chi},\eta).$$

Let $W\in \CW(\pi,\psi), \phi\in \CS(E,\chi)$, as in the proof of Proposition \ref{prop22}, we have $W_a\in \CW(\pi,\psi_{a\bar a}), \phi_a\in \CS(E,\chi)$, and 
$$\Psi(s,W_a,\phi_a,\Phi,\eta, \omega_{\mu,\psi^{-1}_{a\bar a},\chi})=\mu(a)^{-1}\eta(a)^{-1}|a|_E^{-s-1/2}\Psi(s,W,\theta,\eta,\omega_{\mu,\psi^{-1},\chi}).$$
Similarly, we have 
$$\Psi(1-s,W_a,\phi_a,\hat \Phi,\eta^*, \omega_{\mu,\psi^{-1}_{a\bar a},\chi})=\mu( a)^{-1}\eta(\bar a)|a|_E^{s-3/2}\Psi(1-s,W,\theta,\hat \Phi,\eta^*,\omega_{\mu,\psi^{-1},\chi}).$$
Thus $$\gamma(s,\pi,\omega_{\mu,\psi^{-1}_{a\bar a},\chi},\eta,\psi')=\eta(a\bar a)|a|_E^{2s-1}\gamma(s,\pi,\omega_{\mu,\psi^{-1},\chi},\eta,\psi').$$
Denote the Fourier transform of $\Phi$ with respect to $\psi'_b$ by $\hat \Phi_b$, i.e., 
$$\hat \Phi_b(x,y)=\int_{F^2}\Phi(u,v)\psi'_b(yu-xv)dudv.$$
Then $\hat \Phi_b(x,y)=\hat \Phi(bx,by)$. Thus we have
$$f(1-s,h,\hat \Phi_b,\eta^*)=\eta(b)|b|_E^{s-1}f(1-s,h,\hat \Phi,\eta^*).$$
By the functional equation, we get
$$\gamma(s,\pi,\omega_{\mu,\psi^{-1},\chi},\eta,\psi'_b)=\eta(b)|b|_E^{s-1}\gamma(s,\pi,\omega_{\mu,\psi^{-1},\chi},\eta,\psi').$$
\end{proof}
\noindent\textbf{Remark:} Since the dependence of the $\gamma$-factor on $\psi'$ is very simple, we usually take $\psi'=\psi$ and drop it from the notation.

Lemma \ref{lem28} shows that the dependence of the $\gamma$-factor on $\psi_b$ for $b\in \Nm_{E/F}(E^\times)$ is also simple. Next, we need to consider the relation between $\gamma(s,\pi,\omega_{\mu,\psi^{-1},\chi},\eta)$ and $\gamma(s,\pi,\omega_{\mu,\psi^{-1}_\kappa,\chi},\eta)$ for $\kappa\in F^\times-\Nm_{E/F}(E^\times)$ if bote $\gamma$-factors are defined.

 For any $\kappa\in F^\times$, we denote $\alpha(\kappa)=\diag(\kappa,1)\in \GL_2(F)$. We can check if $h=\begin{pmatrix}a& b\\ c& d \end{pmatrix}\in H$, then 
$$h^\kappa:=\alpha(\kappa)h\alpha(\kappa)^{-1}=\begin{pmatrix} a & \kappa b\\ \kappa^{-1}c & d \end{pmatrix}\in H.$$ 
We have $(h_1h_2)^\kappa=h_1^\kappa h_2^\kappa.$ 
We now fix $\kappa\in F^\times-\Nm_{E/F}(E^\times)$.

For an irreducible smooth representation $(\pi,V)$ of $H$, let $(\pi^\kappa, V^\kappa)$ be the representation of $H$ defined by
$$V^\kappa=V, \textrm{ and }\pi^\kappa(h)=\pi(h^\kappa).$$

\begin{lem}
As a vector space, we have $$(V^\kappa)_{N,\psi_\kappa}=V_{N,\psi}.$$
In particular,  $\CW(\pi,\psi)\ne 0$ if and only if $\CW(\pi^\kappa,\psi_\kappa)\ne 0$.
\end{lem}
\begin{proof} Consider the space
$$V^\kappa(N,\psi_\kappa)=\pair{\pi^\kappa(n)v-\psi_\kappa(n)v| n\in N,v\in V}.$$
Since $\pi^\kappa(n)v-\psi_\kappa(n)v=\pi(n^\kappa)v-\psi(n^\kappa)v$, we have $V^\kappa(N,\psi_\kappa)=V(N,\psi)$, and thus $V^\kappa_{N,\psi_\kappa}=V_{N,\psi}$.
\end{proof}



\begin{lem}
We have 
$$(\omega_{\mu,\psi^{-1},\chi})^\kappa\cong \omega_{\mu,\psi_\kappa^{-1},\chi}.$$
\end{lem}
\begin{proof}
We can check the identity map $$(\omega_{\mu,\psi_\kappa^{-1},\chi}, \CS(E,\chi))\ra ((\omega_{\mu,\psi^{-1},\chi})^\kappa, \CS(E,\chi))$$
defines an isomorphism. \end{proof}

 We define an involution $h\mapsto h^\delta$ on $H$ by
$$\begin{pmatrix}a& b\\ c& d \end{pmatrix}^\delta=\begin{pmatrix}\bar a& -\bar b\\-\bar c& \bar d \end{pmatrix}.$$
This is the so-called MVW-involution. For MVW-involution for more general unitary groups, see \cite{MVW}, p91 or \cite{KS}, p270. For a smooth irreducible admissible representation $\pi$ of $H$, it is known that $\tilde \pi\cong \pi^\delta$, where $\pi^\delta(h)=\pi(h^\delta)$.

Let $w=\begin{pmatrix}&1\\ -1& \end{pmatrix}$. Let $h^\alpha=w^{-1}h^\delta w$, and let $\pi^\alpha(h)=\pi(h^\alpha)$. Then $\pi^\alpha\cong \tilde \pi.$ Explicitly, we have
$$\begin{pmatrix}a& b\\ c& d \end{pmatrix}^\alpha=\begin{pmatrix}\bar d& \bar c\\\bar b& \bar a \end{pmatrix}.$$
Recall that we use $\bar N$ to denote lower triangular unipotent subgroup of $H=\RU(1,1)$ and for $x\in F$, we denote
$$\bar n (x)=\begin{pmatrix} 1&\\ x&1 \end{pmatrix}.$$
\begin{cor}{\label{cor211}}
We have 
$$\Hom_N(\pi,\psi)\ne 0 \textrm{ if and only if }\Hom_{\bar N}(\tilde \pi,\psi)\ne 0.$$
\end{cor}
\begin{proof}
For $n(b)\in N$ with $b\in F$, we have $n(b)^{\alpha}=\bar n(b)$. Thus $ \Hom_N(\pi,\psi)=\Hom_{\bar N}(\pi^\alpha, \psi).$ The assertion follows from the fact that $\tilde \pi \cong \pi^\alpha$.
\end{proof}
 
 The proof of the following theorem uses the Gelfand-Kazhdan's method, and is inspired by the proof of Theorem 4.4.2. of \cite{Bu}, the uniqueness of Whittaker functional for $\GL_2$.
 \begin{thm}
 Let $(\pi,V)$ be an irreducible admissible representation of $U(1,1)$ such that $\CW(\pi,\psi)\ne 0$ and $\CW(\pi,\psi_\kappa)\ne 0$, then 
 $$\pi\cong \pi^\kappa.$$
 \end{thm}
 \begin{proof} The condition means that $\CW(\pi,\psi')\ne 0$ for every nontrivial additive character $\psi'$ of $F$, which is equivalent to $\CW(\tilde \pi,\psi')\ne 0$ for every nontrivial additive character $\psi'$ of $F$ by the isomorphism $\tilde \pi \cong \pi^\delta$. Since each irreducible admissible representation $\pi$ is a contragradient of another irreducible admissible representation, it suffices to show that 
 $$\tilde \pi \cong (\tilde \pi)^\kappa\cong (\pi^\alpha)^\kappa.$$
 Let $h^\beta=(h^\kappa)^\alpha$, and $\pi^\beta(h)=\pi(h^\beta)$. Then $\pi^\beta=(\pi^\alpha)^\kappa$. Thus it suffices to show that $\tilde\pi\cong \pi^\beta$.
 
 Define $h^\theta=(h^{-1})^\beta=(h^\beta)^{-1}$. Explicitly, we have 
 $$\begin{pmatrix} a&b \\ c& d\end{pmatrix}^\theta=\begin{pmatrix} a&-\kappa^{-1}c \\ -\kappa b& d\end{pmatrix}.$$
We have $(h_1h_2)^\theta=h_2^\theta h_1^\theta$ and $(h^\theta)^\theta=h$. Moreover, we have $N^\theta=\bar N$ and $\bar N^\theta=N$. 

The assumption implies that $\Hom_N(\tilde \pi,\psi)\ne 0$ and $\Hom_N(\tilde \pi,\psi_\kappa)\ne 0$. We fix a nonzero element $\mu\in \Hom_N(\tilde\pi,\psi_\kappa)$. By Corollary \ref{cor211}, the condition $\Hom_N(\tilde \pi,\psi)\ne 0$ is equivalent $\Hom_{\bar N}(\pi,\psi)\ne 0$. We fix a nonzero element $\lambda\in \Hom_{\bar N}(\pi,\psi)$. 

The dual map $\mu^*$ of $\mu:\tilde V\ra \BC$ defines a map $\mu^*:\BC\ra \tilde V^*$. The smoothness of $\mu$ shows that the image of $\mu^*$ is contained in $\tilde {\tilde V}\cong V$. We define a distribution $T$ on $\CS(H)$ by $$T(f)=\lambda\circ \pi(f)\circ \mu^* \in \End(\BC)\cong\BC.$$

Let $r$ be the right translation and $l$ be the left translation action on $\CS(H)$, i.e., $r(h_0)f(h)=f(hh_0), l(h_0)f(h)=f(h_0^{-1}h)$.
Consider the action $\rho$ of $\bar N\times N$ on $\CS(H)$ defined by 
$$\rho_{\bar n_1,n_2}f(h)=(l(\bar n_1)r(n_2)f)(h)=f(\bar n_1^{-1}hn_2), \bar n_1\in \bar N, n_2\in N.$$
For all $\bar n_1\in \bar N,n_2\in N$, we have 
\begin{align*}T(\rho_{\bar n_1,n_2}f)&=\lambda\circ \int_{H}f(\bar n_1^{-1}h n_2)\pi(h)dh \circ \mu^*\\
&=\lambda\circ \int_H f(h) \pi(\bar n_1)\circ \pi(h) \circ \pi(n_2^{-1}) dh \circ \mu^*\\
&=\lambda\circ \pi(\bar n_1)\circ \pi(f)\circ \pi(n_2^{-1})\circ \mu^*.\end{align*} 
Since $\lambda\circ \pi(\bar n_1)=\psi(\bar n_1)\lambda$ and $\pi(n_2^{-1})\circ \mu^*=\psi_\kappa(n_2)\mu^*$, we get
\begin{equation}{\label{eq29}}
T(\rho_{\bar n_1,n_2}f)=\psi(\bar n_1)\psi_\kappa(n_2)T(f), \forall \bar n_1\in \bar N,n_2\in N.
\end{equation}
 
 We claim that 
 $$(*) \quad T(f)=T(f^\theta), \textrm{ for all }f\in \CS(H),$$ where $f^\theta(h)=f(h^\theta)$. Let $A$ be the torus of $H$ temporarily. By the exact sequence
 $$0\ra \CS(\bar NAN)\ra \CS(H)\ra \CS(\bar N wAN)\ra 0,$$
to prove Claim $(*)$, it suffices to consider $f\in \CS(\bar NAN)$ and $f\in \CS(\bar NwAN)$ separately. 
 
 We first assume that $f\in \CS(\bar NAN)$. We define a function $G_f\in \CS(A)$ by 
\begin{equation}G_f(a)=\int_{F\times F}f(\bar n(-x_1)an(x_2))\psi^{-1}(x_1)\psi^{-1}_\kappa(x_2)dx_1dx_2.\end{equation}
 The assignment $f\mapsto G_f$ defines a surjection $\CS(\bar NAN)\ra \CS(A)$. We define a distribution $\tau$ on $\CS(A)$ by 
 $$\tau(G)=T(f)$$
 if $G=G_f$ for some $f\in \CS(\bar NAN)$. To show that $\tau$ is well-defined, we need to check that if $G_f=0$ then $T(f)=0$. By (\ref{eq29}), we have $T(f)=\psi^{-1}(n_1)\psi^{-1}_\kappa(n_2)T(\rho_{\bar n_1,n_2}f)$ for any $\bar n_1\in \bar N, n_2\in N$, thus for any compact subsets $C_1,C_2\subset F$, we have 
 $$T(f)=\frac{1}{\Vol(C_1\times C_2)}\int_{C_1\times C_2}\psi^{-1}(x_1)\psi^{-1}_\kappa(x_2)T(\rho_{\bar n(x_1),n(x_2)}f)dx_1dx_2=\frac{1}{\Vol(C_1\times C_2)}T(f'),$$
 where $f'\in \CS(\bar NAN)$ is defined by $$f'(h)=\int_{C_1\times C_2}\psi^{-1}(n_1)\psi^{-1}_\kappa(n_2)f(\bar n(-x_1)hn(x_2))dx_1dx_2. $$
 If $C_1,C_2$ are large enough, we have $f'(a)=G_f(a)=0, \forall a\in A$ by assumption. Thus $f'=0$ and $T(f)=0$. Then $\tau$ is well-defined. To show $T(f)=T(f^\theta)$ it suffices to show that $G_f=G_{f^\theta}$. We have
 \begin{align*}
 G_{f^\theta}(a)&=\int_{F\times F}f^\theta(\bar n(-x_1)an(x_2))\psi^{-1}(x_1)\psi_\kappa^{-1}(x_2)dx_1dx_2\\
 &=\int_{F^\times F}f(n(x_2)^\theta a^\theta \bar n(-x_1)^\theta)\psi^{-1}(x_1)\psi^{-1}_\kappa(x_2)dx_1dx_2\\
 &=\int_{F^\times F}f(\bar n(-\kappa x_2)a^\theta \bar n(\kappa^{-1}x_1 ))\psi^{-1}(x_1)\psi^{-1}_\kappa(x_2)dx_1dx_2
 \end{align*}
 Let $x_1'=\kappa x_2$ and $x_2'=\kappa^{-1} x_1$, we see that $dx_1dx_2=dx_1'dx_2'$, and the last expression of the above integral becomes
 $$ \int_{F^\times F}f(\bar n(-x_1')a^\theta \bar n(x_2'))\psi^{-1}_\kappa(x_2')\psi^{-1}(x_1')dx_1'dx_2'=F_f(a^\theta).$$
 Thus we get
 $$G_{f^\theta}(a)=G_f(a^\theta).$$
 Since for $a\in A$, we have $a=a^\theta$, we get $G_f=G_{f^\theta}$. This completes the proof of $T(f)=T(f^\theta)$ when $f\in \CS(\bar NAN)$. 
 
 Next we consider the case $f\in \CS(\bar NwAN)=\CS(wB)$, where $B$ is the Borel subgroup of $H$. Similarly as above, we have a surjection map $\CS(\bar NwAN)\mapsto \CS(wA)$, $f\mapsto G_f$, where $G_f$ is defined by
 $$G_f(wa)=\int_N \psi_\kappa^{-1}(n)f(wan)dn.$$
 
  We can define a distribution $\tau$ on $\CS(wA)$ by $\tau(G)=T(f)$ if $G=G_f$ for some $f\in \CS(\bar NwAN)$. Since $T(r(n)f)=\psi(n)T(f)$, a similar argument as above will show that $\tau$ is well-defined. We claim that $\tau\equiv 0$ on $\CS(wA)$. 
  
 Since $T(l(\bar n)f)=\psi(\bar n)T(f)$, we get $G_{l(\bar n)f}-\psi(\bar n)G_f\in \Ker(\tau)$ for any $\bar n\in \bar N$. For $f\in \CS(\bar NwAN)$, and $\bar n\in \bar N$, we define $\Phi_{f,\bar n}=G_{l(\bar n)f}-\psi(\bar n)G_f$.  We have 
 \begin{align*}
 \Phi_{f,\bar n}(wa)&=\int_{N}\psi^{-1}_\kappa(n')f(\bar n^{-1}wa n')dn'-\psi(\bar n)G_f(wa)\\
 &=\int_N \psi_\kappa^{-1}(n')f(wa (wa)^{-1}\bar n^{-1}wa n')dn'-\psi(\bar n)G_f(wa)\\
 &=(\psi_\kappa((wa)^{-1}\bar n^{-1}wa)-\psi(\bar n))G_f(wa).
 \end{align*}
 We suppose that
 $$wa=\begin{pmatrix} & b\\-\bar b^{-1} & \end{pmatrix}, b\in E^\times, \bar n= \begin{pmatrix} 1& \\x & 1 \end{pmatrix},$$
 then 
 $$(wa)^{-1}\bar n^{-1}wa= \begin{pmatrix} 1& b\bar b x \\& 1 \end{pmatrix} .$$
 Thus 
 $$(\psi_\kappa((wa)^{-1}\bar n^{-1}wa)-\psi(\bar n))=\psi(\kappa b\bar bx)-\psi(x).$$
 Since $\kappa\notin \Nm(E^\times)$, we have $\kappa b\bar b\ne 1$. Since $\psi$ is continuous, for a small open compact neighborhood $D$ of $wa$, we can find an $\bar n=\bar n(x)$ such that $ (\psi_\kappa((wa)^{-1}\bar n^{-1}wa)-\psi(\bar n))$ is a nonzero constant $c_D$, for all $wa\in D$. Let $f_D\in \CS(\bar NwAN)$ such that $G_{f_D}$ is the characteristic function of $D$, then we have 
 $$\Phi_{f,\bar n}|_D=c_DG_{f_D}.$$
 The space $\CS(wA)$ is spanned by $G_{f_D}$, and thus can be spanned by $\Phi_{f,\bar n}|_D$. This shows that $\CS(wA)\subset \Ker(\tau)$, i.e., $T(f)=\tau(G_f)=0$ for all $f\in \CS(\bar NwAN)$. A similar consideration will show that $T(f^\theta)=0$. Thus $T(f)=T(f^\theta)$ for $f\in \CS(wB)$.
 
 This finishes the proof of the claim $(*)$.

 We define a bilinear form $\bB$ on $\CS(H)$ by 
 $$\bB(f*\phi)=T(f*\check \phi),$$
 where $\check \phi(h)=\phi(h^{-1})$, and $*$ means convolution. By Claim $(*)$, we get
\begin{equation}{\label{eq211}}\bB(f,\phi)=T(f*\check\phi)=T((f*\check \phi)^\theta)=T(\check\phi^\theta*f^\theta)=\bB(\phi^\beta,f^\beta).\end{equation}
 
 Define a map 
 $$\lambda: \CS(H)\ra \tilde\pi$$
 $$f\mapsto \lambda_f$$
 where $\lambda_f\in \tilde V$ is defined by $\lambda_f(v)=\lambda(\pi(f)v)$. Let  $J(\lambda)$ be the kernel of the map $\lambda$, i.e., $J(\lambda)=\wpair{f\in \CS(H): \lambda_f=0}$
 
 Similarly, we define $\mu: \CS(H)\ra \tilde {\tilde \pi} \cong \pi$ by $\mu_f(\tilde v)=\mu(\tilde \pi(f)v)$ and $J(\mu)$ to be the kernel of $\mu$. 
 
It is easy to see that
$$J(\lambda)=\wpair{f\in \CS(H)| \bB(f,\phi)=0,\forall \phi\in \CS(H)},$$
and 
$$J(\mu)=\wpair{\phi\in \CS(H)| \bB(f,\phi)=0,\forall f\in \CS(H)},$$

By (\ref{eq211}), we have 
$$J(\lambda)=\wpair{f\in \CS(H)| B(\phi^\beta, f^\beta)=0, \forall \phi\in \CS(H)}=J(\mu)^\beta.$$

Let $(r,\CS(H))$ be the right translation action of $H$ on $\CS(H)$, then $\lambda:(r,\CS(H))\ra \tilde \pi$ is an intertwining surjection with kernel $J(\lambda)$. Let $(r^\beta,\CS(H))$ be the representation of $H$ on $\CS(H)$ defined by $r^\beta(h)f=r(h^\beta)f$. Then the map $\mu:\CS(H)\ra \pi$ defines an intertwining surjection $\mu:(r^\beta,\CS(H))\ra \pi^\beta$, with kernel $J(\mu)$.

The assignment $f\mapsto f^\beta$ defines an isomorphism $(r,\CS(H))\ra (r^\beta,\CS(H))$. Since $J(\mu)=J(\lambda)^\beta$, we have the following commutative diagram
$$\xymatrix{0\ar[r]&J(\lambda)\ar[d]\ar[r]&S(H)\ar[r]^{\lambda} \ar[d]&\tilde \pi\ar@{-->}[d]^{\beta}\ar[r]& 0\\
0\ar[r]&J(\mu)\ar[r]&S(H)\ar[r]^{\mu}&  \pi^\beta\ar[r]&0
}$$ 
and thus we get an isomorphism $\tilde \pi\ra \pi^\beta$. This completes the proof.
 \end{proof}

\begin{cor}{\label{cor213}}
 Let $\pi$ be an irreducible smooth representation of $H$ which is both $\psi$- and $\psi_\kappa$-generic. For any $W\in \CW(\pi,\psi)$, define a function $W^\kappa$ on $H$ by $W^\kappa(h)=W(h^\kappa)$. Then $W^\kappa\in \CW(\pi,\psi_\kappa)$. Moreover, the assignment $W\mapsto W^\kappa$ defines a bijection from $\CW(\pi,\psi)$ to $\CW(\pi,\psi_\kappa)$. 
\
\end{cor}
\begin{proof} 
 By the above theorem, we have an isomorphism $^\kappa: (\pi,V)\ra (\pi^\kappa, V)$. Since the map $^\kappa$ is intertwining, we have
\begin{equation}{\label {kappa is intertwining}}
(\pi(h)v)^\kappa=\pi^\kappa(h)v^\kappa, h\in H, v\in V.
\end{equation}
 Let $\lambda\in \Hom_N(\pi, \psi)$ be a nonzero element. We define $\lambda^\kappa$ by $\lambda^\kappa(v)=\lambda(v^\kappa)$. For $n\in N$, by Eq.(\ref{kappa is intertwining}) we have
$$\lambda^\kappa(\pi(n)v)=\lambda((\pi(n)v)^\kappa)=\lambda(\pi^\kappa(n)v^\kappa)=\lambda(\pi(n^\kappa)v^\kappa)=\psi_\kappa(n)\lambda(v^\kappa)=\psi_\kappa(n)\lambda^\kappa(v),$$
thus $\lambda^\kappa\in \Hom_N(\pi,\psi_\kappa)$.

Since ${}^\kappa$ is an isomorphism, for each $W\in \CW(\pi, \psi)$, we can take $v\in V$ such that $W(h)=\lambda(\pi(h)v^\kappa)$. Then
$$W^\kappa(h)=\lambda(\pi(h^\kappa)v^\kappa)=\lambda((\pi(h)v)^\kappa)=\lambda^\kappa(\pi(h)v).$$
Thus $W^\kappa\in \CW(\pi,\psi_\kappa)$. The ``moreover" part follows from the fact that ${}^\kappa$ is an isomorphism.
\end{proof}

 \begin{prop}{\label{prop214}}
 Let $(\pi,V)$ be an irreducible admissible representation of $H$ such that $\CW(\pi,\psi^{-1})\ne 0$ and $\CW(\pi,\psi_\kappa^{-1})\ne 0$. The notations $\mu,\chi,\psi,\eta$ are as usual. Let $\kappa$ be an element of $F^\times-\Nm(E^\times)$. 
 Then $$L(s,\pi,\omega_{\mu,\chi,\psi_\kappa},\eta)=L(s,\pi,\omega_{\mu,\chi,\psi},\eta)$$
 and 
 $$\gamma(s,\pi,\omega_{\mu,\chi,\psi_\kappa},\eta,\psi')=\eta(\kappa)|\kappa|_F^{2s-1}\gamma(s,\pi,\omega_{\mu,\chi,\psi},\eta,\psi').$$
 \end{prop}
 \begin{proof}
 We fix an isomorphism $\phi\mapsto \phi^\kappa$ of $(\omega_{\mu,\psi^{-1},\chi})^{\kappa}\ra \omega_{\mu,\psi^{-1}_\kappa,\chi} .$ Then we have 
 $$\omega_{\mu,\psi^{-1},\chi}(h^\kappa)\phi=\omega_{\mu,\psi^{-1}_\kappa,\chi}(h)\phi^\kappa.$$
 
 For $W\in \CW(\pi,\psi)$, by Corollary \ref{cor213}, we have $W^\kappa\in \CW(\pi,\psi_\kappa)$. For $\Phi\in \CS(F^2)$, and $\Re(s)>>0$, we have
\begin{align*}& \quad \Psi(s,W^\kappa, \phi^\kappa, \Phi,\eta,\omega_{\mu,\psi^{-1}_\kappa,\chi})\\
&=\int_{R\setminus H}W^\kappa(h)(\omega_{\mu,\psi^{-1}_\kappa,\chi}(h)\phi^\kappa)(1)f(s,h,\Phi)dh\\
&=\int_{R\setminus H}W(h^\kappa)(\omega_{\mu,\psi^{-1},\chi}(h^\kappa)\phi)(1)f(s,h,\Phi)dh\\
&=\int_{R\setminus H}W(h)(\omega_{\mu,\psi^{-1},\chi}(h)\phi)(1)f(s,h^{\kappa^{-1}},\Phi)dh.
 \end{align*}
 Write $h=t(a)g$ with $g\in \SL_2(F)$, we have $h^{\kappa^{-1}}=t(a)g^{\kappa^{-1}}$. Denote $\Phi^\kappa:=\diag(\kappa,1)\Phi$. Then
  \begin{align*}&\quad f(s,h^{\kappa^{-1}}, \Phi)\\
  &=\eta(a)|a|^s\int_{F^\times}[\left(\diag(\kappa^{-1},1)g\diag(\kappa,1) \right)\Phi](0,r)\eta(r)|r|_E^s d^*r\\
  &=\eta(a)|a|^s\int_{F^\times}[\left(g\diag(\kappa,1) \right)\Phi](0,r)\eta(r)|r|_E^s d^*r\\
  &=f(s,h, \Phi^{\kappa}).\end{align*}
  Thus 
  \begin{equation}{\label{eq213}}\Psi(s,W^\kappa,\phi^\kappa,\Phi,\eta,\omega_{\mu,\psi^{-1}_\kappa,\chi})=\Psi(s,W,\theta,\Phi^\kappa,\eta,\omega_{\mu,\psi^{-1},\chi}).\end{equation}
 By Corollary \ref{cor213} and the fact that $\phi\mapsto \phi^\kappa$ is an isomorphism, $I(s,\pi,\omega_{\mu,\psi_\kappa^{-1},\chi},\eta)$ is generated by $\Psi(s,W^\kappa,\phi^\kappa,\Phi,\eta,\omega_{\mu,\psi^{-1}_{\kappa a\bar a},\chi})$, for $W\in \CW(\pi,\psi_{ a\bar a}), \phi\in \CS(E,\chi), a\in E^\times$, $\Phi\in \CS(F^2)$. Since $\Phi\mapsto \Phi^\kappa$ is an isomorphism, by Eq.(\ref{eq213}), we have 
 $I(s,\pi,\omega_{\mu,\psi^{-1}_\kappa,\chi},\eta)=I(s,\pi,\omega_{\mu,\psi^{-1},\chi},\eta)$. Consequently, we get
 $$ L(s,\pi,\omega_{\mu,\psi^{-1}_\kappa,\chi},\eta)=L(s,\pi,\omega_{\mu,\psi^{-1},\chi},\eta).$$
 Since $\widehat{\Phi^\kappa}=|\kappa|_F^{-1}\diag(\kappa^{-1},\kappa^{-1})\hat\Phi^{\kappa}$, we get
  $$f(1-s,h,\widehat{\Phi^\kappa},\eta^*)=\eta(\kappa)^{-1}|\kappa|_F^{1-2s}f(1-s,h,\hat \Phi^\kappa,\eta^*).$$
 From the functional equation, we have
 $$\gamma(s,\pi,\omega_{\mu,\psi^{-1}_\kappa,\chi},\eta)=\eta(\kappa)|\kappa|_F^{2s-1}\gamma(s,\pi,\omega_{\mu,\psi^{-1},\chi},\eta).$$
 This finishes the proof.
 \end{proof}
 We can combine Lemma \ref{lem28} and Proposition \ref{prop214} together to get
  \begin{cor}
 For $a\in F^\times$,  suppose that both $L(s,\pi,\omega_{\mu,\psi^{-1},\chi},\eta)$ and $L(s,\pi,\omega_{\mu,\psi^{-1}_a,\chi})$ are defined, then we have
  $$L(s,\pi,\omega_{\mu,\psi^{-1}_a,\chi},\eta)=L(s,\pi,\omega_{\mu,\psi^{-1},\chi},\eta)$$
  and $$\gamma(s,\pi,\omega_{\mu,\psi^{-1}_a,\chi},\eta)=\eta(a)|a|_F^{2s-1}\gamma(s,\pi,\omega_{\mu,\psi^{-1},\chi},\eta).$$
  \end{cor}

\subsection{Unramified caculation} The unramified calculation is done in \cite{GRS} and we include their result here for completeness.

Let $E/F$ be unramified, $\psi$ be an additive character with conductor $\CO_F$. Let $p$ be a prime element of $F$. Let $(\pi,V)$ be an irreducible unramified representation of $H$ such that $\CW(\pi,\psi)\ne 0$. We can assume that $\pi$ is an irreducible unramified component of $\Ind_B^H(\nu)$ for an unramified quasi-character $\nu$ of $E^\times$. Let $W$ be the Whittaker function associated with the spherical vector, normalized by the condition $W(k)=1$ for $k\in K$. The explicit Casselman-Shalika formula reads: $$W\left(t(a)\right)=0, ~ |a|_E >1,$$
and 
$$W_\pi(t(p^n))=|p^n|_F\frac{\nu(p)^n-\nu(p)^{-n-1}}{1-\nu(p)^{-1}}, n\ge 0,$$
where $p$ is a prime element in $E$.

 Let $\mu$ be the unique unramified character of $E^\times$ such that $\mu|_{F^\times}=\epsilon_{E/F}$. Note that $\mu(p)=-1$. Let $\sigma=\omega_{\mu,\psi^{-1},1}$. Then $\sigma$ is a component of $\Ind_B^H(\mu)$. If we take $\phi$ to be the characteristic function of $\CO_E$ and let $W_\sigma(h)=(\omega_{\mu,\psi^{-1},1}(h)\phi)(e)$. Then a direct calculation or an application of the above Casselman-Shalika formula show that
 $$W_\sigma(t(a))=0, |a|> 1$$
 and 
 $$W_\sigma(t(p^n))=(-1)^n|p|_F^n$$
 
 Let $\eta$ be an unramified character of $E^\times$. Then the condition $\omega_\pi\omega_\sigma\eta|_{E^1}=\nu\mu\eta|_{E^1}=1$ is automatic. Let $\Phi\in\CS(F^2)$ be the characteristic function of $\CO_F\oplus \CO_F$, let $f(s,h,\Phi,\eta)$ be the element in $\Ind_B^H(\eta|~|^{s-1/2})$ as before. 
 
 Write $H=BK$. For $h=um(a)k$, we have $dh=|a|_E^{-1}du dadk$. We have 
 \begin{align*}
&\quad \Psi(s,W_\pi,W_\sigma,\Phi,\eta)\\
&=\int_{K}\int_{E^1\setminus E^\times}W_\pi(t(a)k)W_\sigma(t(a)k)f(t(a)k)|a|^{-1}dadk\\
&=\int_{E^1\setminus E^\times} W_\pi(t(a))W_\sigma(t(a))\eta(a)|a|^{s-1}f(1)da\\
&=\frac{f(s,1,\Phi,\eta)}{1-\nu(p)^{-1}}\Vol(E^1\setminus \CO_E^\times)\left(\sum_{n\ge 0} (-\eta\nu(p))^nq_E^{-ns}-\sum_{n\ge 0} \nu(p)^{-1}(-\eta\nu^{-1}(p))^nq_E^{-ns}\right)\\
&=\Vol(E^1\setminus \CO_E^\times) \frac{f(s,1,\Phi,\eta)}{1-\nu(p)^{-1}}\left( \frac{1}{1-\mu\eta\nu(p)q_E^{-s}}-\frac{\nu(p)^{-1}}{1-\mu\eta\nu^{-1}(p)q_E^{-s}}\right)\\
&=\Vol(E^1\setminus \CO_E^\times) \frac{f(s,1,\Phi,\eta)}{1-\nu(p)^{-1}}\frac{(1-\nu(p)^{-1})(1-\eta(p)q_E^{-s})}{(1-\mu\eta\nu(p)q_E^{-s})(1-\mu\eta\nu^{-1}(p)q_E^{-s})}\\
&=\Vol(E^{1}\setminus \CO_E^\times)f(s,1,\Phi,\eta) L_E(s,\eta)^{-1}L_E(s,\mu\eta \nu)L_E(s,\mu\eta\nu^{-1}).
\end{align*}
By definition, we have 
\begin{align*}
f(s,1,\Phi,\eta)&=\int_{F^\times}\Phi(0,a)|a|^s_E\eta(a)d^\times a\\
&=\sum_{n\ge 0}|p|_E^{ns}\eta(a)^n\Vol(\CO_F^\times)\\
&=\Vol(\CO_F^\times)L_E(s,\eta).
\end{align*}

Thus, we have
$$\Psi(s,W_\pi,\phi,\Phi,\eta)=cL_E(s,\mu\eta\nu)L_E(s,\mu\eta\nu^{-1}),$$
where $c$ is a nonzero constant which depends on the measure.
\begin{prop} Let the notations be as above, we have 
$$L(s,\pi,\omega_{\mu,\psi^{-1},1},\eta)=L_E(s,\mu\eta\nu)L_E(s,\mu\eta\nu^{-1}),$$
and
$$\epsilon(s,\pi,\omega_{\mu,\psi^{-1},1},\eta)=1.$$
\end{prop}

\section{A local converse theorem for $\RU(1,1)$}
In this section, we will slightly modify Baruch's method, see \cite{Ba1} and \cite{Ba2}, to give a proof of the local converse theorem for $\RU(1,1)$, when $E/F$ is unramified or $E/F$ is ramified but the characteristic of the residue field is not 2,  using the $\gamma$-factors $\gamma(s,\pi,\omega_{\mu,\psi^{-1},\chi},\eta)$ we defined in the last section. 

For the field extension $E/F$, there is an associated integer $i=i_{E/F}$ defined as follows. If $E/F$ is unramified, then $i_{E/F}=0$. For $E/F$ ramified, take an element $x\in \CO_E$ such that $\CO_E=\CO_F[x]$, and define $i=v_E(\bar x-x)$, where $v_E$ is the valuation of $E$. The integer $i$ is the smallest integer such that the ramification group $G_i$ is trivial, see Chapter IV, $\S1$ of \cite{Se}. It is known that $i\ne 1$ if and only if the residue field $\CO_F/\CP_F$ has characteristic $ 2$, see Chapter IV, $\S2$ of \cite{Se}.

\subsection{Howe vectors}
The Howe vectors for the groups $\GL_n$, $\RG\Sp_4$ and $\RU(2,1)$ are defined in Baruch's thesis \cite{Ba2}, and the Howe vectors for $\RU(2,1)$ can also be found in \cite{Ba1}. The $\RU(1,1)$ version we will use can be defined similarly.

Let $p_E$ be a prime element of $E$, $\CO_E$ (resp. $\CO_F$) be the integer ring of $E$ (resp. $F$), and $\CP_E$ (resp. $\CP_F$) be the maximal ideal of $E$ (resp. $F$). Let $q_F=\# (\CO_F/\CP_F)$ and $q_E=\#(\CO_E/\CP_E)$.
 
Thus, if $E/F$ is unramified, we can take $p_E=p_F$ and we have $q_E=q_F^2$ and $\CP_F=\CP_E\cap \CO_F$. If $E/F$ is ramified, we have $p_F=p_E^2 u$ for some $u\in \CO_E^\times$, $q_E=q_F$ and $\CP_F=\CP_E^2\cap \CO_F$.

Let $\psi$ be an unramified additive character of $F$ and let $\psi_E$ be the additive character of $E$ defined by $\psi_E=\psi\circ(\frac{1}{2} \tr_{E/F})$. Thus for $x\in F\subset E$, we have $\psi_E(x)=\psi(x)$. 

For a positive integer $m$, let $K_m=(1+\textrm{Mat}_{2\times 2}(\CP_E^m))\cap \RU(1,1)$ if $E/F$ is unramified, and $K_m=(1+\textrm{Mat}_{2\times 2}(\CP_E^{2m}))\cap \RU(1,1)$. We can write $K_m=(1+\textrm{Mat}_{2\times 2}((\CP_F\CO_E)^m))$ uniformly. Set
$$d_m=\begin{pmatrix} p_F^{-m}&\\ & p_F^{m}\end{pmatrix}.$$
Let $J_m=d_mK_md_m^{-1}$. For $k=(k_{il})\in K$, we have
\begin{equation}{\label {kj}} j:=d_mkd_m^{-1}=\begin{pmatrix} k_{11}& p_F^{-2m}k_{12}\\ \bar p_F^{2m}k_{21} & k_{22}\end{pmatrix}.\end{equation}
 Thus
$$J_m=\begin{pmatrix} 1+\CP_E^m& \CP_E^{-m}\\ \CP_E^{3m}& 1+\CP_E^m\end{pmatrix}\cap \RU(1,1), \textrm{ if } E/F \textrm{ is unramified},$$
and 
$$J_m=\begin{pmatrix} 1+\CP_E^{2m}& \CP_E^{-2m}\\ \CP_E^{6m}& 1+\CP_E^{2m}\end{pmatrix}\cap \RU(1,1), \textrm{ if } E/F \textrm{ is ramified}.$$

For a subgroup $A$ of $H=\RU(1,1)$, denote $A_m=A\cap J_m$. Note that, in any case, we have
$$N_m=\begin{pmatrix}1& \CP_F^{-m}\\ & 1 \end{pmatrix}.$$

Let $\tau_m$ be the character of $K_m$ defined by
$$\tau_m(k)=\psi_E( p_F^{-2m}k_{1,2}), k=(k_{i,l})\in K_m.$$
By our assumption on $\psi_E$, it is easy to see that $\tau_m$ is indeed a character on $K_m$. Define a character $\psi_m$ on $J_m$ by
$$\psi_m(j)=\tau_m(d_m^{-1}jd_m), j\in J_m.$$
We have $\psi_m(j)=\psi_E(j_{1,2})$ for $j=(j_{i,l})\in J_m$. Thus $\psi_m$ and $\psi$ agree on $N_m$.

Let $(\pi,V)$ be a $\psi$-generic representation. We fix a Whittaker functional. Let $v\in V$ be such that $W_v(1)=1$. For $m\ge 1$, as \cite{Ba1} and \cite{Ba2}, we define
\begin{equation}{\label{eq32}}v_m=\frac{1}{\Vol(N_m)}\int_{N_m}\psi(n)^{-1}\pi(n)vdn.\end{equation}
Let $L\ge 1$ be an integer such that $v$ is fixed by $K_L$.

\begin{lem}{\label {lem31}} We have
\begin{enumerate}
\item $W_{v_m}(1)=1.$
\item If $m\ge L$, $\pi(j)v_m=\psi_m(j)v_m$ for all $j\in J_m$.
\item If $k\le m$, then
$$v_m=\frac{1}{\Vol(N_m)}\int_{N_m}\psi(n)^{-1}\pi(n)v_kdn.$$
\end{enumerate}
\end{lem}
\begin{proof}
The proof is the same as in the $\RU(2,1)$ case, see Lemma 5.2 of \cite{Ba1}. We give some details of the proof of (2) here. Let 
$$\tilde v_m=\frac{1}{\vol(J_m)}\int_{J_m}\psi_m(j)^{-1} \pi(j)v dj.$$
We have the Iwahori decomposition $J_m= N_m \cdot \bar B_m$ with unique expressions, where $\bar B_m=\bar B\cap J_m$. For $j=n \bar b$, then $dj=d\bar b d n$. Note that $ \bar B_m \subset K_m\subset K_L$ for $m\ge L$, and thus $v$ is fixed by $\bar B_m$. Then
\begin{align*}
\tilde v_m&=\frac{1}{\vol(N_m)}\frac{1}{\vol (\bar B_m)}\int_{N_m}\int_{\bar B_m} \psi_m(n\bar b)^{-1}\pi(n \bar b)v d\bar bdn\\
 &=\frac{1}{\vol(N_m)}\int_{N_m} \psi_m(n)\pi(n)v dn\\
 &=\tilde v_m.
\end{align*}
It is clear that $\pi(j)\tilde v_m=\psi_m(j)\tilde v_m$. The assertion follows.
\end{proof}
 The vectors $\wpair{v_m}_{m\ge L}$ are called Howe vectors.

Let $w=\begin{pmatrix} &1\\ -1& \end{pmatrix}$ be the unique nontrivial Weyl element of $\RU(1,1)$.
\begin{lem}{\label{lem32}}
For $m\ge L$, let $v_m$ be Howe vectors defined as above, and let $W_{v_m}$ be the Whittaker functions associated to $v_m$. Then
\begin{enumerate}
\item $W_{v_m}(t(a))\ne 0 \textrm{ implies } a\bar a\in 1+\CP_F^m.$
\item  $W_{v_m}(t(a)w)\ne 0$ implies $a\bar a\in \CP_F^{-3m}$.
\end{enumerate}
\end{lem}
\begin{proof}
(1) For $x\in P_F^{-m}$, we have $n(x)\in  N_m$. On the other hand, we have
$$t(a)n(x)=n(a\bar ax)t(a).$$
By Lemma \ref{lem31}, we have
$$\psi_m(n(x))W_{v_m}(t(a))=\psi(a\bar ax)W_{v_m}(t(a)).$$
Since $\psi_m(n(x))=\psi(x)$, if $W_{v_m}(t(a))\ne 0$, we have $(1-a\bar a)x\in \Ker(\psi)$ for any $x\in \CP_F^{-m}$. Thus $a\bar a\in 1+\CP_F^m$.  

(2) For $x\in \CP_F^{3m}$, we have $\bar n(-x)\in \bar N_m:=\bar N\cap J_m$. From the relation
$$t(a)w\bar n(-x)=n(a\bar a x)t(a)w,$$
we have
$$W_{v_m}(t(a)w)=\psi(a\bar a x)W_{v_m}(t(a)w).$$
Thus $W_{v_m}(t(a)w)\ne 0$ implies $\psi(a\bar a x)=1$ for all $x\in \CP_F^{3m}$, i.e., $a\bar a\in \CP_F^{-3m}$.
\end{proof}
Denote the central character of $\pi$ by $\omega_\pi$.
\begin{lem}{\label{lem33}}
\begin{enumerate}
\item Suppose that $E/F$ is unramified. For $a\in E^\times$, $\Nm_{E/F}(a)=a\bar a\in 1+\CP_F^m$ if and only if $a\in E^1(1+\CP_E^m)$. Thus, for $m\ge L$
$$ W_{v_m}(t(a))=\left\{\begin{array}{lll}\omega_\pi(z), &\textrm{ if } a=zu, \textrm{ for }z\in E^1, u\in 1+\CP_E^m, \\ 0, & \textrm{ otherwise}.    \end{array}\right. $$
\item Suppose that $E/F$ is ramified and $m\ge i_{E/F}$. For $a\in E^\times$, $\Nm_{E/F}(a)=a\bar a\in 1+\CP_F^m$ if and only if $a\in E^1(1+\CP_E^{2m-i_{E/F}+1})$. 
Thus, if the residue characteristic is not $2$, then for $m\ge L\ge i_{E/F}=1$, we have
$$ W_{v_m}(t(a))=\left\{\begin{array}{lll}\omega_\pi(z), &\textrm{ if } a=zu, \textrm{ for }z\in E^1, u\in 1+\CP_E^{2m}, \\ 0, & \textrm{ otherwise}.    \end{array}\right. $$
\end{enumerate}
\end{lem}
\begin{proof}
We first assume that $E/F$ is unramified, it is known that $\Nm_{E/F}(1+\CP_E^m )=1+\CP_F^m$, see \cite{Se} Chapter V $\S2$, for example. It is clear that $\Nm(a)\in 1+\CP_F^m$ for $a\in E^1(1+\CP_E^m)$. On the other hand, if $a\bar a\in 1+\CP_F^m$, we can find $b\in 1+\CP_E^m$ such that $a\bar a= \Nm_{E/F}(b)=b\bar b$. Thus $b/a\in E^1$ and $a=b\cdot (b/a)\in E^1(1+\CP_E^m)$. If $a\notin E^1(1+\CP_E^m)$, then $a\bar a \notin 1+\CP_F^m$, and thus $W_{v_m}(t(a))=0$ by Lemma \ref{lem32}. If $a\in E^1(1+\CP_E^m)$, we write $a=zu$. Then $t(z)$ is in the center of $H$ and $t(u)\in J_m$ by the definition of $J_m$. By Lemma \ref{lem31} (2), we get
$$W_{v_m}(t(a))=\omega_{\pi}(z)W_{v_m}(t(u))=\omega_{\pi}(z).$$

If $E/F$ is ramified, then it is known that, for $m\ge i_{E/F}$, we have $\Nm_{E/F} (1+\CP_E^{2m-i_{E/F}+1})=1+\CP_F^m,$ see Corollary 3 of $\S3$, Chapter V of \cite{Se}. The same argument as above shows that $a\bar a\in 1+\CP_F^{m}$ if and only if $a\in E^1( 1+\CP_E^{2m-i_{E/F}+1})$. Now we assume the residue characteristic is not 2, and hence $i_{E/F}=1$. Thus $a\bar a\in 1+\CP_F^{m}$ if and only if $a\in E^1( 1+\CP_E^{2m})$. If $a\notin E^1(1+\CP_E^{2m})$, then $a\bar a \notin 1+\CP_F^m$, and thus $W_{v_m}(t(a))=0$ by Lemma \ref{lem32} (2). If $a\in E^1(1+\CP_E^{2m})$, write $a=zu$ for $z\in E^1 $ and $u\in 1+\CP_E^{2m}$. By the definition of $J_m$, we have $t(u)\in J_m$. Thus  we get $$W_{v_m}(t(a))=\omega_\pi(z)W_{v_m}(t(u))=\omega_\pi(z),$$ by Lemma \ref{lem31} (2) again.
\end{proof}

In the following of this section, we will assume that $E/F$ is unramified, or $E/F$ is ramified, but the residue characteristic is not 2. We fix two irreducible smooth $\psi$-generic representations $(\pi,V_\pi)$ and $(\pi',V_{\pi'})$ of $\RU(1,1)$ with the same central character. Fix $v\in V_\pi$ and $v'\in V_{\pi'}$ with $W_v(1)=1=W_{v'}(1)$ and positive integers $m$, we can define Howe vectors $v_m$ and $v_m'$. Let $L\ge 1$ be an integer such that $v$ and $v'$ are fixed by $K_L$ under the action of $\pi$ and $\pi'$ respectively. 

\begin{prop}{\label{prop34}}
Let $n_0\in N$. Then
\begin{enumerate}
\item $W_{v_m}(g)=W_{v_m'}(g)$ for all $g\in B, m\ge L;$
\item If $n_0\in N_m$, then $W_{v_m}(twn_0)=\psi(n_0)W_{v_m}(tw)$ for all $t\in T, m\ge L;$
\item If $n_0\notin N_m$, then $W_{v_m}(twn_0)=W_{v'_m}(twn_0)$ for all $t\in T$, $m\ge 3L.$
\end{enumerate}
\end{prop}
\begin{proof}

(1) Since $B=NT$, it suffices to show that $W_{v_m}(t)=W_{v_m'}(t)$ for all $t\in T$. Write $t=t(a)$ with $a\in E^\times$. This follows from Lemma \ref{lem33} directly.

(2) For $n_0\in N_m$, we have $\pi(n)v_m=\psi(n_0)v_m$ by Lemma \ref{lem31}. The assertion is clear.

(3) By Lemma \ref{lem31} (3), we have
\begin{equation}{\label{eq34}}W_{v_m}(twn_0)=\Vol(N_m)^{-1}\int_{N_m}W_{v_L}(twn_0n)\psi^{-1}(n)dn.\end{equation}
We have a similar relation for $W_{v_m'}$. Thus it suffices to show $W_{v_L}(twn_0n)=W_{v_L'}(twn_0n)$ for all $n\in N_m$. Let $n\in N_m$ and $n'=n_0n$. Since $n_0\notin N_m,n\in N_m$, we get $n'\notin N_m$.

 Suppose $n'=n(x)$. Then $n'\notin N_m$ is equivalent to $x\notin \CP_F^{-m}$ or $x^{-1}\in \CP_F^m$. 
We have
$$\begin{pmatrix} &1 \\ -1& \end{pmatrix}\begin{pmatrix} 1&x \\ &1 \end{pmatrix}=\begin{pmatrix}- x^{-1}& 1\\ &-x \end{pmatrix}\begin{pmatrix} 1& \\ x^{-1} &1\end{pmatrix}$$
Let $b=\begin{pmatrix}- x^{-1}& 1\\ &-x \end{pmatrix},j=\begin{pmatrix} 1& \\ x^{-1} &1\end{pmatrix}$. Then $b\in B$. Since $m\ge 3L$, we have $x^{-1}\in \CP_F^m\subset \CP_F^{3L}$, i.e., $ j\in J_L$. Thus by Lemma \ref{lem31}, we get
$$W_{v_L}(twn')=W_{v_L}(tbj)=W_{v_L}(tb).$$
Since $tb\in B$, we have $W_{v_L}(tb)=W_{v_L'}(tb)$ by Part (1). Thus $W_{v_L}(twn')=W_{v_L'}(twn')$. This completes the proof.
\end{proof}

\subsection{Howe vectors for Weil representations}{\label{sec32}} 

Let $\chi$ be a character of $E^1$, we define $\deg(\chi)$ to be the smallest integer $i$ such that $\chi|_{E^1\cap(1+\CP_E^i)}= 1$, i.e., the integer $i$ such that $\chi|_{E^1\cap(1+\CP^i_E)}=1$ but $\chi|_{E^1\cap (1+\CP_E^{i-1})}\ne 1$. If $\chi=\chi_0$ is the trivial character, we define $\deg(\chi_0)=0$.
Let $\chi$ be a character of $E^1$, $a\in E^\times$, we consider the integral
\begin{equation}F_\chi(a)=\int_{E^1}\psi(\tr_{E/F}(au))\chi(u)du.\label{eq34*}\end{equation}
Note that as a character on $E$, $\psi(\tr_{E/F}(\cdot)) $ is unramified if $E/F$, and has conductor $\fD_{E/F}^{-1}$ if $E/F$ is ramified. In any case, $\psi(\tr_{E/F}(\cdot))$ is trivial on $\CO_E$.
\begin{lem}{\label{lem35}}
Let $\chi_0$ be the trivial character of $E^1$, and $a\in E^\times$ with $|a|\le q_E^n$. If $n\le 0$, then 
$$F_{\chi_0}(a)\ne 0.$$
\end{lem}
\begin{proof}
 If $n\le 0$, we get $a\in \CO_E$, thus $\psi(\tr_{E/F}(au))=1$ for all $u\in E^1$. Then $$F_{\chi_0}(a)=\int_{E^1}\psi(\tr_{E/F}(au))du=\vol(E^1)\ne 0.$$
\end{proof}

\begin{lem}{\label{lem36}} Suppose that $a\in E^\times$ with $|a|_E=q_E^n$, and $\chi$ is a character of $E^1$ with $h=\deg(\chi)\ge 1.$
\begin{enumerate}
\item  If $h> \max\{n,1\}$, we have  $F_\chi(a)=0$.
\item  If $n>0$, there is a character $\chi$ of $E^1$ with $\deg(\chi)\le n$ such that $F_\chi(a)\ne 0$. 
\end{enumerate}
\end{lem}
\begin{proof}
(1) We first suppose that $n\le 0$, then $\psi(\tr_{E/F}(au))=1$ for all $u\in E^1$ and thus $F_\chi(a)=\int_{E^1}\chi(u)du=0$ since $\chi$ is nontrivial on $E^1$. Next, we assume that $1\le n \le\deg(\chi)$. We have
$$F_\chi(a)=\sum_{u_0\in E^1/(E^1\cap (1+\CP_E^n))}\int_{u_1\in E^1\cap(1+\CP_E^n)}\psi(\tr_{E/F}(au_0u_1))\chi(u_0u_1)du_1.$$
For $u_0\in E^1,u_1\in E^1\cap(1+\CP_E^n)$, we have $au_0u_1-au_0=au_0(u_1-1)\in \CO_E$, since $a\in \CP_E^{-n}$. Thus $\psi(\tr_{E/F}(au_0u_1))=\psi(\tr_{E/F}(au_0))$. Hence
$$F_\chi(a)=\sum_{u_0\in E^1/(E^1\cap(1+\CP_E^n))}\psi(\tr_{E/F}(au_0))\chi(u_0)\int_{E^1\cap (1+\CP_E^n)}\chi(u_1)du_1.$$
Since $\deg(\chi)> n$, $\chi|_{E^1\cap(1+\CP_E^n)}\ne 1$, we get $\int_{E^1\cap (1+\CP_E^n)}\chi(u_1)du_1=0 $. Thus $F_\chi(a)=0$.

(2) Consider the function $u\mapsto\psi(\tr_{E/F}(au))$ on $E^1$. Since this function is nonzero, there must be a character $\chi$ of $E^1$ such that its $\chi$-th Fourier coefficient
$$F_\chi(a)=\int_{E^1}\psi(\tr_{E/F}(au))\chi(u)du$$
is nonzero. By (1), such a character $\chi$ must satisfy $\deg(\chi)\le n$.
\end{proof}

We will consider the Howe vectors for the representation $(\omega_{\mu,\psi^{-1},\chi},\CS(E,\chi))$ for a character $\mu$ of $E^\times$ such that $\mu|_{F^\times}=\epsilon_{E/F}$, and a character $\chi$ of $E^1$. If $E/F$ is unramified and $m$ is an integer such that $m\ge \deg(\chi)$, we define $\phi^{m,\chi}$ by $\supp(\phi^{m,\chi})=E^1(1+\CP_E^m)$, and 
$$\phi^{m,\chi}(zu)=\chi(z), z\in E^1, u\in 1+\CP_E^m.$$
Note that $\phi^{m,\chi}$ is well-defined since $\chi_{E^1\cap(1+\CP_E^m)}=1.$ 

If $E/F$ is ramified with the residue characteristic not 2, for an integer $m$ with $2m\ge \deg(\chi)$, we define $\phi^{m,\chi}\in \CS(E,\chi)$ by $\supp(\phi^{m,\chi})=E^1(1+\CP_E^{2m})$, and 
$$\phi^{m,\chi}(zu)=\chi(z), z\in E^1, u\in 1+\CP_E^{2m}.$$
Note that $\phi^{m,\chi}$ is well-defined since $\chi_{E^1\cap(1+\CP_E^{2m})}=1.$ 

Let $\fD_{E/F}=\CP_E^d$ be the different of $E/F$, for some $d\ge 0$. 
\begin{prop}{\label{prop37}}  Let $m$ be an integer such that $m\ge d$.
\begin{enumerate}
\item For $n\in N_m$, we have
$$\omega_{\mu,\psi^{-1},\chi}(n)\phi^{m,\chi}=\psi^{-1}(n)\phi^{m,\chi}.$$
\item  For $\bar n \in \bar N_{m}=\bar N\cap J_m$, we have $$\omega_{\mu,\psi^{-1},\chi}(\bar n)\phi^{m,\chi}=\phi^{m,\chi}.$$
\end{enumerate}
\end{prop}
\begin{proof} 
(1) For $n=n(b)\in N_m$, we have
\begin{align*}
\quad\omega_{\mu,\psi^{-1},\chi}(n(b))\phi^{m,\chi}(x)
=\psi^{-1}(bx\bar x))\phi^{m,\chi}(x).
\end{align*}
For $x\in \supp(\phi^{m,\chi})$, we have $x\bar x\in 1+\CP_F^m$ in either case by Lemma \ref{lem33}. For $n(b)\in N_m$, we have $b\in \CP_F^{-m}$, and thus $bx\bar x -b\in \CO_F$. Since $\psi$ is unramified, we get $\psi^{-1}(bx\bar x)=\psi^{-1}(b)$.
Thus $$\omega_{\mu,\psi^{-1},\chi}(n(b))\phi^{m,\chi}=\psi^{-1}(b)\phi^{m,\chi}.$$

(2) For simplicity, we will write $\omega$ for $\omega_{\mu,\psi^{-1},\chi}$. For $\bar n\in \bar N_m$, we can write $\bar n= w^{-1}n(b)w$ with $b\in \CP_F^{3m}$.
Let $\phi'=\omega(w)\phi^{m,\chi}$. We have
\begin{align*}
\phi'(x)=\gamma_{\psi^{-1}}\int_{E}\phi^{m,\chi}(y)\psi^{-1}(\tr(\bar x y))dy.
\end{align*}
It is clear that $\supp \phi' \subset \CP_E^{-m}$ in the unramified case and $\supp \phi'\subset \CP_E^{-2m-d}$ in the ramified case. For $x\in \supp \phi'$, we have $bx\bar x\in \CO_F$ in either case since $m\ge d$. Thus
\begin{align*}
\omega(n(b))\phi'(x)&=\psi(b x\bar x)\phi'(x)=\phi'(x),
\end{align*}
i.e., $\omega(n(b))$ fixes $\phi'$. Then
$$\omega(\bar n)\phi^{m,\chi}=\omega(w^{-1})\omega(n(b))\omega(w)\phi^{m,\chi}$$
$$=\omega(w^{-1}) \omega(n(b)) \phi'=\omega(w^{-1}) \phi'=\omega(w^{-1})\omega(w)\phi^{m,\chi}=\phi^{m,\chi}.$$
This completes the proof.
\end{proof}

\begin{prop}{\label{prop38}}
Given a positive integer $m$. 
\begin{enumerate}
\item Suppose that $E/F$ is unramified. Then for $a\in \CP_E^{-m}$, we have
$$(\omega_{\mu,\psi^{-1},\chi}(t(a)w)\phi^{m,\chi})(1)=\mu(a)|a|^{1/2}\gamma_{\psi^{-1}}\vol (1+\CP_E^m)\vol(E^1\cap (1+\CP_E^m)) F_\chi(\bar a),$$
where $\gamma_{\psi^{-1}}$ is the Weil index, and $F_\chi(\bar a)= \int_{E^1} \chi(z)\psi(\tr( \bar a z ))dz$, which is defined in Eq.$(\ref{eq34*})$.
\item Suppose that $E/F$ is ramified. Then for $a\in \CP_E^{-2m}$, we have
$$(\omega_{\mu,\psi^{-1},\chi}(t(a)w)\phi^{m,\chi})(1)=\mu(a)|a|^{1/2}\gamma_{\psi^{-1}}\vol (1+\CP_E^{2m})\vol(E^1\cap (1+\CP_E^{2m})) F_\chi(\bar a),$$
where $\gamma_{\psi^{-1}}$ is the Weil index, and $F_\chi(\bar a)= \int_{E^1} \chi(z)\psi(\tr( \bar a z ))dz$, which is defined in Eq.$(\ref{eq34*})$.
\end{enumerate}
\end{prop}
\begin{proof}
We only consider the unramified case. The ramified case can be done similarly.

We have
\begin{align*}
&\quad \omega(t(a)w)\phi^{m,\chi}(1)\\
&=\mu(a)|a|^{1/2}\omega(w)\phi^{m,\chi}(a)\\
&= \mu(a)|a|^{1/2}\gamma_{\psi^{-1}}\int_E \phi^{m,\chi}(y)\psi^{-1}(-\tr(\bar a y))dy\\
&=\mu(a)|a|^{1/2}\gamma_{\psi^{-1}}\int_{(E^1\cap (1+\CP_E^m))\setminus (1+\CP_E^m)} \int_{E^1} \chi(z)\psi(\tr( \bar a z u))dz du,
\end{align*}
Since $a\in \CP_E^{-m}$, for $z\in E^1, u\in 1+\CP_E^m$, we have $ \bar a z u-\bar a z\in \CO_E$. Thus
\begin{align*}
&\quad \omega(t(a)w)\phi^{m,\chi}(1)\\
&=\mu(a)|a|^{1/2}\gamma_{\psi^{-1}}\vol (1+\CP_E^m)\vol(E^1\cap (1+\CP_E^m))^{-1} \int_{E^1} \chi(z)\psi(\tr( \bar a z ))dz \\
&=\mu(a)|a|^{1/2}\gamma_{\psi^{-1}}\vol (1+\CP_E^m)\vol(E^1\cap (1+\CP_E^m))^{-1} F_\chi(\bar a),
\end{align*}
where $F_\chi(\bar a)= \int_{E^1} \chi(z)\psi^{-1}(\tr( \bar a z ))dz$. This completes the proof of (1).\end{proof}

\subsection{A local converse theorem}\label{sec33}
The following functions and computations are done by Baruch, see page 335 of \cite{Ba1}.

For a set $A\subset F$, let $\Upsilon_A$ the characteristic function of $A$. Let 
$$\Phi_{i,l}(x,y)=\Upsilon_{P_F^i}(x)\Upsilon_{1+P_F^l}(y).$$
Then
$$\hat \Phi_{i,l}(x,y)=q_F^{-i-l}\psi(-x)\Upsilon_{P_F^{-l}}(x)\Upsilon_{P_F^{-i}}(y).$$
For $x\in F$, recall we denote 
$$n(x)=\begin{pmatrix} 1& x\\ & 1\end{pmatrix},$$
and $$\bar n(x):=\begin{pmatrix} 1& \\ x& 1\end{pmatrix}.$$
Let $\eta$ be a quasi-character of $E^\times$, and $l=l(\eta)=\deg(\eta)$ be the conductor of $\eta$, i.e., $\eta(1+P_E^{\ell})=1$ but $\eta(1+P_E^{l-1})\ne 1$. We have
\begin{equation}{\label{eq34}}f(\bar n(x),\Phi_{i,l},\eta,s)=\left\{\begin{array}{lll} q_F^{-l}& \textrm{ if } |x|_F\le q_F^{-i}, \\ 0 & \textrm{ otherwise;}    \end{array} \right.\end{equation}
and 
\begin{equation}{\label{eq35}}
f(wn(x), \hat \Phi_{i,l},\eta^*, 1-s)=\left\{\begin{array}{lll} q_F^{-l-i}\int_{|y|_F\le q_F^l}\psi_F(y)\eta^{-1}(y)|y|^{1-s} d^*y& \textrm{ if } |x|_F\le q_F^{i-l}, \\  q_F^{-l-i}\int_{|y|_F\le q_F^i|x|_F}\psi_F(y)\eta^{-1}(y)|y|^{1-s} d^*y, &\textrm{ if } |x|_F> q_F^{i-l}.  \end{array} \right.
\end{equation}
where the equality in (\ref{eq35}) is in the sense of continuation.
Let $$c(s,\eta,\psi)=\int_{|y|\le q_F^{l(\eta)}}\eta^{-1}(y)\psi_F(y)|y|^sd^*y,$$
which has no zeros or poles except for a finite number of $q_F^{-s}$, see \cite{Ba1} and the references given there.

\begin{thm}{\label{thm39}}
Let $E/F$ be a quadratic extension of $p$-adic fields and $H=\RU(1,1)(F)$. We exclude the case when $E/F$ is ramified and the residue characteristic is $2$. Let $\psi$ be an additive unramified character of $F$ and $\mu$ of a character of $E^\times$ such that $\mu|_{F^\times}=\epsilon_{E/F}$. Let $\pi,\pi'$ be two $\psi$-generic irreducible representations of $H$ with the same central character.
\begin{enumerate}
\item If $\gamma(s,\pi, \omega_{\mu,\psi^{-1},\chi},\eta)=\gamma(s,\pi', \omega_{\mu,\psi^{-1},\chi},\eta)$ for all characters $\chi$ of $E^1$, and all quasi-characters $\eta$ with 
$$\omega_\pi\cdot \mu\chi\cdot \eta|_{E^1}=1,$$
where $\omega_\pi$ is the central character of $\pi$, then $\pi$ is isomorphic to $\pi'$.
\item  For a character $ \chi$ of $E^1$, there exists an integer $l_0=l_0(\pi,\pi', \mu,\chi)$ such that for any quasi-character $\eta$ of $E^\times$ which satisfies $$\omega_\pi\cdot \mu\chi\cdot \eta|_{E^1}=1,$$ and $l(\eta)>l(\eta_0)$, then $\gamma(s,\pi, \omega_{\mu,\psi^{-1},\chi}, \eta)=\gamma(s,\pi',\omega_{\mu,\psi^{-1},\chi}, \eta)$.
\end{enumerate}
\end{thm}

\begin{proof}
In the following, we assume that $E/F$ is unramified. The ramified case can be dealt with similarly. Write $\omega_{\mu,\psi^{-1},\chi}$ as $\omega$ for simplicity.

We fix vectors $v\in V, v'\in V'$ such that $W_{v}(1)=1=W_{v'}(1)$. Let $L_1$ be an even integer such that $v,v'$ are fixed by $K_{L_1}$.  Let $L=3L_1$, $C=3L/2$,  $\eta$ be a quasi-character of $E^\times$ with $l=\deg(\eta)$, $\chi$ be a character of $E^1$ and $m$ be an integer such that $m\ge \max\wpair{C,\deg(\mu), l=\deg( \eta), \deg(\chi)}$.  Since $\deg(\chi)\le m$,  we can consider the function $\phi^{m,\chi}\in \CS(E,\chi)$ defined in $\S$\ref{sec32}, i.e.,  $\phi^{m,\chi}(zu)=\chi(z)$ for $z\in E^1$ and $u\in 1+\CP_E^m$, and zero otherwise. Let $i$ be an integer such that $i\ge 3m> m+l,$ we then have a function $\Phi_{i,l}\in \CS(F^2)$.

We now compute the integral
$$\Psi(s,W_{v_m},\phi^{m,\chi},\Phi_{i,l},\eta)=\int_{R\setminus H}W_{v_m}(h)\omega(h)\phi^{m,\chi}(1)f(s,h,\Phi_{i,l},\eta)dh.$$

We will take the integral $\Psi(s,W_{v_m}, \phi^{m,\chi},\Phi_{i,l})$ on the open dense set $N\setminus NT\bar N$ of $N\setminus H$. For $h=n(y)t(a)\bar n(x)\in NT\bar N$, the Haar measure can be taken as $dh=|a|_E^{-1}dyd^*a dx$. Thus by Eq.(\ref{eq34}), we get
\begin{align*}
&\qquad \Psi(s,W_{v_m},\phi^{m,\chi},\Phi_{i,l},\eta)\\
&=\int_{E^1\setminus E^\times}\int_F W_{v_m}(t(a)\bar n(x))(\omega(t(a)\bar n(x))\phi^{m,\chi})(1)\eta(a)|a|^{s-1}f(s,\bar n(x),\Phi_{i,l},\eta)dxd^*a\\
&=q_{F}^{-l}\int_{E^1\setminus E^\times}\int_{|x|_F\le q_F^{-i}}W_{v_m}(t(a)\bar n(x))(\omega(t(a)\bar n(x))\phi^{m,\chi})(1)\eta(a)|a|^{s-1}dxd^*a.
\end{align*}

Since $i\ge 3m$,  $|x|_F\le q_F^{-i}$ implies that $\bar n(x)\in J_m$. Thus by Lemma \ref {lem31} and Proposition \ref{prop37}, we have 
$$W_{v_m}(t(a)\bar n(x))= W_{v_m}(t(a)), \omega(t(a)\bar n(x))\phi^{m,\chi}=\omega(t(a))\phi^{m,\chi}.$$
On the other hand, by the definition of $\phi^{m,\chi}$, we have
$$\omega(t(a))\phi^{m,\chi}(1)=\mu(a)|a|^{1/2}\phi^{m,\chi}(a)=\left\{\begin{array}{lll}\mu(zu)\chi(z), & \textrm{ if } a=zu,\textrm{ for } z\in E^1, u\in 1+\CP_E^m, \\ 0, & \textrm { otherwise }\end{array} \right. $$
Combining these facts and Lemmas \ref{lem33}, we get
\begin{align*}
&\quad \Psi(s,W_{v_m},\phi^{m,\chi},\Phi_{i,l})\\
&=q_F^{-l-i} \int_{E^1\setminus E^\times}W_{v_m}(t(a))(\omega(t(a))\phi^{m,\chi})(1)\eta(a)|a|^{s-1}d^*a\\
&=q_F^{-l-i}\int_{E^1\setminus E^1(1+\CP_E^m)}W_{v_m}(t(a))(\omega(t(a))\phi^{m,\chi})(1)\eta(a)d^*a\\
&=q_F^{-l-i} \int_{E^1\cap(1+\CP_E^m)\setminus (1+\CP_E^m)}W_{v_m}(t(a))\omega(t(a))\phi^{m,\chi}(1)\eta(a)d^*a\\
&=q_F^{-l-i} \int_{E^1\cap(1+\CP_E^m)\setminus (1+\CP_E^m)}\mu(a)\eta(a)d^*a.
\end{align*}
Since $m\ge \deg(\mu\eta)$, we have
\begin{align}
&\quad \Psi(s,W_{v_m},\phi^{m,\chi},\Phi_{i,l})\label{eq36}\\
&=q_F^{-l-i} \int_{E^1\cap (1+\CP_E^m)\setminus (1+\CP_E^m)}\mu(a)\eta(a)d^*a \nonumber\\
&=q_F^{-l-i}\vol(1+\CP_E^m)\vol(E^1\cap (1+\CP_E^m))^{-1},\nonumber
\end{align}
which is a nonzero constant. Note that this finishes the proof of Proposition \ref{prop22} when $E/F$ is unramified.

The same calculation works for $\Psi(s,W_{v'_m},\phi^{m,\chi},\Phi_{i,l})$. Thus, we get
\begin{equation}{\label{eq37}}\Psi(s,W_{v_m'},\phi^{m,\chi},\Phi_{i,l},\eta)=\Psi(s,W_{v_m},\phi^{m,\chi},\Phi_{i,l})=q_F^{-l-i}\vol(1+\CP_E^m)\vol(E^1\cap (1+\CP_E^m))^{-1},\end{equation}
for $ i\ge 3m, m\ge \max\wpair{C,\deg(\mu), l=\deg(\eta),\deg(\chi)}.$

Let $W=W_{v_m}$ or $W_{v_m'}$. We compute the integral $\Psi(1-s,W,\phi^{m,\chi}, \hat \Phi_{i,l})$ on the dense set $N\setminus NTwN$ of $N\setminus H$.  Since $i\ge m+l $, then $|x|_F\le q_E^{m}$ implies that $|x|_F\le q^{i-l}_F$. Then by Eq.(\ref{eq35}), we have
\begin{align}
&\quad\Psi(1-s,W,\phi^{m,\chi}, \hat \Phi_{i,l}) \nonumber\\
&=\int_{E^1\setminus E^\times}\int_{F }W(t(a)wn(x))(\omega(t(a)wn(x) )\phi^{m,\chi})(1) f_s(t(a)wn(x), \hat \Phi_{i,l}, \eta^*)\nonumber\\
&=q_F^{-l-i}c(1-s, \eta, \psi)\int_{E^1\setminus E^\times}\int_{|x|\le q_F^m}W(t(a)wn(x) )\omega(t(a)wn(x))\phi^{m,\chi}(1)\eta^*(a)|a|^{-s}dxd^*a \label{eq38}\\
&+\int_{E^1\setminus E^\times}\int_{|x|_F>q_F^m} W(t(a)wn(x) )\omega(t(a)wn(x))\phi^{m,\chi}(1) f(1-s, t(a)wn(x), \hat \Phi_{i,l})dxd^*a \label{eq39}
\end{align}
By Proposition \ref{prop34}, we have $W_{v_m}(t(a)wn(x))=W_{v_m'}(t(a)wn(x))$ for $|x|>q_F^m$. Thus the expressions (\ref{eq39}) for $W=W_{v_m}$ and $W=W_{v'_m}$ are the same.

 For $|x|_F\le q_F^m$, we have $n(x)\in N_m$. By Lemma \ref{lem32} and Proposition \ref{prop37},  we have
$$W_{v_m}(t(a)wn(x))=\psi(x)W_{v_m}(t(a)w), \textrm{ and }\omega(t(a)wn(x))\phi^{m,\chi}(1)=\psi^{-1}(x)\omega(t(a)w)\phi^{m,\chi}(1).$$

Thus the expression (\ref{eq38}) can be simplified to
$$q_F^{-l-i+m}c(1-s, \eta, \psi)\int_{E^1\setminus E^\times}W(t(a)w)\omega(t(a)w)\phi^{m,\chi}(1)\eta^*(a)|a|^{-s}d^*a.  $$
By the above discussion, we get
\begin{align}
&\Psi(s,W_{v_m},\phi^{m,\chi},\hat \Phi_{i,l})-\Psi(s,W_{v_m'},\phi^{m,\chi},\hat \Phi_{i,l} )\label{eq310}\\
=&q_F^{-l-i+m}c(1-s, \eta^*, \psi)\int_{E^1\setminus E^\times}(W_{v_m}(t(a)w)-W_{v_m'}(t(a)w))\omega(t(a)w)\phi^{m,\chi}(1)\eta^*(a)|a|^{-s}d^*a. \nonumber
\end{align}
By the local functional equation, Eq.(\ref{eq37}) and (\ref{eq310}), we get
\begin{align}
&d_m(s,\eta,\psi)(\gamma(s,\pi,\omega_{\mu,\psi^{-1},\chi},\eta)-\gamma(s,\pi',\omega_{\mu,\psi^{-1},\chi},\eta)) \label{eq311}\\
=&\int_{E^1\setminus E^\times}(W_{v_m}(t(a)w)-W_{v_m'}(t(a)w))\omega(t(a)w)\phi^{m,\chi}(1)\eta^*(a)|a|^{-s}d^*a,\nonumber
\end{align}
where $d_m(s,\eta,\psi)=q_F^{-m}c(1-s,\eta,\psi)^{-1}\vol(1+\CP_E^m)\vol(E^1\cap (1+\CP_E^m))^{-1}$.
Since $m\ge C\ge L$, then by Lemma \ref{lem31} (3), we have 
$$W_{v_m}(t(a)w)-W_{v_m'}(t(a)w)=\frac{1}{\vol(N_m)}\int_{N_m}\psi^{-1}(n)(W_{v_L}(t(a)wn)-W_{v'_L}(t(a)wn))dn.$$
By Proposition \ref{prop34}, if $n\in N_m-N_L$, we have $W_{v_L}(t(a)wn)=W_{v_L'}(t(a)wn)$. Thus 
\begin{align}
&\quad W_{v_m}(t(a)w)-W_{v_m'}(t(a)w)\label{eq312}\\
&=\frac{1}{\vol(N_m)}\int_{N_m}\psi^{-1}(n)(W_{v_L}(t(a)wn)-W_{v'_L}(t(a)wn))dn \nonumber\\
&=\frac{1}{\vol(N_m)}\int_{N_L}\psi^{-1}(n)(W_{v_L}(t(a)wn)-W_{v'_L}(t(a)wn))dn \nonumber\\
&=\frac{\vol(N_L)}{\vol(N_m)}(W_{v_L}(t(a)w)-W_{v_L'}(t(a)w)).\nonumber
\end{align}

If we combine Eq.(\ref{eq311}) and Eq.(\ref{eq312}), we get
\begin{align}
&d_{m,L}(s,\eta,\psi)(\gamma(s,\pi,\omega_{\mu,\psi^{-1},\chi},\eta)-\gamma(s,\pi',\omega_{\mu,\psi^{-1},\chi},\eta)) \label{eq313}\\
=&\int_{E^1\setminus E^\times}(W_{v_L}(t(a)w)-W_{v_L'}(t(a)w))\omega(t(a)w)\phi^{m,\chi}(1)\eta^*(a)|a|^{-s}d^*a,\nonumber
\end{align}
with $d_{m,L}(s,\eta,\psi)=c(s,\eta,\psi)^{-1}\vol(N_L)^{-1}$, for $m\ge \max\wpair{C,\deg(\mu),l=\deg(\eta), \deg(\chi)}$.

 By Lemma \ref{lem32} (2), we have 
$$\supp(W_{v_L}(t(a)w)-W_{v'_L}(t(a)w)\subset \CP_E^{-C}.$$
By Proposition \ref{prop38}, for $a\in \CP_E^{-C}\subset \CP_E^{-m}$, (note that $m\ge C$ by assumption), we have
$$\omega(t(a)w)\phi^{m,\chi}(1)=\mu(a)|a|^{1/2}\vol(1+\CP_E^{m})\vol(E^1\cap (1+\CP_E^m))^{-1}\gamma_{\psi^{-1}}F_\chi(\bar a).$$
Thus Eq.(\ref{eq313}) can be written as 
\begin{align}
&D_L(s,\eta,\psi)(\gamma(s,\pi,\omega_{\mu,\psi^{-1},\chi},\eta)-\gamma(s,\pi',\omega_{\mu,\psi^{-1},\chi},\eta) )\label{eq314}\\
=& \int_{E^1\setminus E^\times}(W_{v_L}(t(a)w)-W_{v_L'}(t(a)w))\mu(a)\eta^*(a)|a|^{-s+1/2}F_\chi(\bar a)d^*a,\nonumber
\end{align}
where $D_L(s,\eta,\psi)=c(s,\eta,\psi)^{-1}\vol(N_L)^{-1}\gamma_{\psi^{-1}}^{-1}$. Note that Eq.(\ref{eq314}) is independent of $m$, and thus it is true for all choices of compatible $\chi, \eta$.

Now we can prove our theorem.

To prove (1), we will show that $W_{v_L}(h)=W_{v_L'}(h)$ for all $h\in H=U(1,1)(F)$. By Proposition \ref{prop34}, it suffices to show that $W_{v_L}(t(a)w)=W_{v_L'}(t(a)w)$ for all $a\in E^\times$. By Lemma \ref{lem32} (2), it suffices to show that
$$W_{v_L}(t(a)w)-W_{v_L'}(t(a)w)=0, \forall a\in E^\times, \textrm{ with }|a|_E\le q_E^C. $$
By assumption of (1) and Eq.(\ref{eq314}), we have
$$\int_{E^1\setminus E^\times}(W_{v_L}(t(a)w)-W_{v_L'}(t(a)w))F_\chi(\bar a) \mu(a)\eta^*(a)|a|^{-s+1/2}d^*a=0,$$
for all compatible $\chi, \eta$. For a fixed $\chi$, if we vary $\eta$, then by inverse Mellin transform on the group $E^1\setminus E^\times$, we get 
$$(W_{v_L}(t(a)w)-W_{v_L'}(t(a)w))F_\chi(\bar a)=0, \forall \chi. $$
By Lemma \ref{lem35} and Lemma \ref{lem36}, for each $a\in \CP_E^{-C}$, we can find a character $\chi$ of $E^1$ with $\deg(\chi)\le C$ such that $F_\chi(\bar a)\ne 0$. Thus we get
\begin{equation}{\label{eq311}}W_{v_L}(t(a)w)=W_{v_L'}(t(a)w), \forall a\in \CP_E^{-C}.\end{equation}
This completes the proof of (1).

Now we prove (2). Let $l_0=l_0(\pi, \pi',\mu,\chi)$ be an integer such that $l_0>\deg(\mu)$, $l_0>C=3L/2$ and
\begin{equation}\label{eq317}W_{v_L}(t(a_0a)w)=W_{v_L}(t(a)w), W_{v_L'}(t(a_0a)w)=W_{v_L'}(t(a)w), \forall a_0\in 1+\CP_E^{l_0}.\end{equation} 
Such $l_0$ exists since the above functions on $a$ are continuous. In fact, if $l_0>L$, the condition Eq.(\ref{eq317}) holds. By definition of $F_\chi$, from $l_0>C$, we can check that 
\begin{equation}F_\chi(\bar a_0 \bar a)=F_\chi(\bar a), \forall a_0\in 1+\CP_E^{l_0} \textrm{ and } a\in \CP_E^{-C}.\end{equation}

 Note that $l_0$ is in fact independent of $\chi$. 
If $l(\eta)>l_0$, from our choice of $l_0$, the right side of Eq.(\ref{eq314}) vanishes. Since $D_L(s,\eta,\psi)$ has no zeros or poles outside of a finite number of $q_E^{-s}$, we get
$$ \gamma(s,\pi, \omega_{\mu,\psi^{-1},\chi},\eta)=\gamma(s,\pi', \omega_{\mu,\psi^{-1},\chi}, \eta).$$
This completes the proof of (2).
\end{proof}
\begin{rem}\upshape \label{rem310}
In the above proof, we omit the details when $E/F$ is ramified. In the ramified case, the analogue of Eq.(\ref{eq36}) is
$$ \Psi(s,W_{v_m}, \phi^{m,\chi}, \Phi_{i,l})=q_F^{-i-l}\vol(1+\CP_E^{2m})\vol(E^1\cap (1+\CP_E^{2m}))^{-1}.$$
It is easy to see that in calculation to derive this formula, we do not have to require that the residue characteristic is not 2. Thus Proposition \ref{prop22} is fully proved.
\end{rem}

\section{ Local Zeta integrals for $\GL(2,F)$}
In this section, we assume$F$ is a $p$-adic local field, and let $H=\GL(2,F)$. We will use the following notations. Let $B$ be the upper triangular subgroup of $H$. Let $B=TN$, with $T$ the torus and $N$ the unipotent subgroup. Let $\bar B$ (resp. $\bar N$) be the opposite of $B$ (resp. $N$). Let $Z$ be the center of $H$ and $R=ZN$.
For $x\in F$, let $$n(x)=\begin{pmatrix}1& x\\ &1 \end{pmatrix}\in N, \bar n(x)=\begin{pmatrix}1&\\ x &1 \end{pmatrix}\in \bar N.$$ For $a,b\in F^\times$, let $t(a,b)=\diag(a,b)$. For $a\in F^\times$, let $t(a)=m(a,1)$. Let $w=\begin{pmatrix}&1\\ -1& \end{pmatrix}$.

\subsection{Weil representation of $\GL_n(F)$}
This section follows from \cite{GR}. Let $X=F^n$ be an $n$-dimensional vector space over $F$ with a fixed choice of basis, and let $W=X\oplus X^*$, where $X^*$ is the dual space of $X$. We also identify $X$ and $X^*$ with $F^n$ and denote the standard paring between $F^n$ by $\pair{~,~}$, where $\pair{v,w}=\sum v_jw_j$ for $v=(v_j)$ and $w=(w_j)$. Let $A((v,v^*),(w,w^*))=\pair{v,w^*}-\pair{v^*,w}$ be the symplectic form on $W$. The group $G=\GL_n(F)$ can be identified with the stabilizer of $X$ and $X^*$ in $\Sp(W)$. The (Schrodinger model of the) Weil representation $\omega_\psi$ of $\Mp(W)$ is realized on $\CS(X)=\CS(F^n)$. The subgroup $G\subset \Sp(W)$ has a lifting to $\Mp(W)$ which is uniquely determined by the following formula:
$$\omega_\psi(g)f(x)=|\det(g)|^{1/2}f({}^t\!g x),~g\in G,x\in F^n.$$ 
We call the representation given by this formulae the standard oscillator representation of $\GL_n(F)$. Any other oscillator representations differ from it by a twist. Let $P$ be the parabolic subgroup with the last row $(0,0,\dots,0,*)$ and $P_n$ be the subgroup of $P$ with the last row $(0,0,\dots,0,1)$. Then there is an isomorphism
$$(\omega_\psi,S(F^n))\ra \ind_{P_n}^G(|\textrm{det}|^{1/2})$$
$$f\mapsto \varphi_f,$$ where $\varphi_f(g)=|\det(g)|^{1/2}f({}^t\!ge_n)=(\omega(g)f)(e_n)$, with $e_n=(0,0,\dots,0,1)^t$. Let $Z$ be the center of $\GL_n(F)$ and $\chi$ be a character of $Z$, let $\omega_{\psi,\chi}$ be the largest quotient of $\omega_\psi$ on which $Z$ acts via $\chi$. To determine $\omega_{\psi, \chi}$, we define 
$$\varphi_{\chi}(g)=\int_{Z}\chi^{-1}(z)\varphi(zg)dz$$
for $\varphi\in \ind_{P_n}^G(|\det|^{1/2})$. Then $(r(z)\varphi)_\chi(g)=\chi(z)\varphi_\chi(g)$, for $z\in Z$, where $r$ is the right translation. 

 The Levi subgroup $M$ of $P$ is isomorphic to $\GL(n-1)\times \GL(1)$. Let $1\otimes \chi$ be the character of $M$ whose restriction to $\GL(n-1)$ is trivial and restriction to $\GL(1)$ is $\chi$. For $h\in \GL(n-1), a\in \GL(1)$, let $p=\diag(h,a)\in P$, we have
$$\varphi_\chi(pg)=|\det(h)|^{1/2}|a|^{-(n-1)/2}\chi(a)\varphi_\chi(g)=(1\otimes \chi)\delta_{P}^{1/2}(p)\varphi_\chi(g),$$
i.e., $\varphi_\chi\in \Ind_{P}^G(1\otimes \chi)$, where $\Ind$ denotes the normalized induced representation.
\begin{prop}
There is an isomorphism $\omega_{\psi,\chi}\cong \Ind_P^G(1\otimes \chi)$, and the map $f\mapsto \varphi_f\mapsto (\varphi_f)_\chi$ defines the projection $(\omega_\psi, \CS(F^n))\cong (r, \ind_{P_n}^G(|\det|^{1/2}))\ra (\omega_{\psi,\chi}\cong \Ind_P^G(1\otimes \chi)).$
\end{prop}
\begin{cor}
Suppose that $n=2$. If $\chi\ne |~|^{\pm 1}$, then $\omega(\psi,\chi)$ is irreducible.
\end{cor}
This is well-known, see \cite{BZ2} for example.

\subsection{A different model of the Weil representation for $\GL(2)$} In this section, we consider a different model of the Weil representation for $\GL_2$, which will give us the familiar Weil representation formulas.
 
Let $X=F^2$ and let $e_1,e_2$ be the standard basis of $X$, let $e_1^*,e_2^*$ be the corresponding dual basis. Let $Y=Fe_1\oplus Fe_2^*$ and $Y^*=Fe_1^*\oplus Fe_2$. Then $W=Y\oplus Y^*$ is a complete polarization of $W$. There is an isomorphism 
$$I: \CS(X)\ra \CS(Y)$$
defined by the partial Fourier transform
$$I(f)(xe_1+ye_2^*)=\int_{F}\psi^{-1}(yz)f(xe_1+ze_2)dz.$$

From the isomorphism, the Weil representation $(\omega_\psi, \CS(X))$ can be realized on $\CS(Y)$, which is also denoted by $\omega_\psi$ by ambiguity notation.

For $\phi=I(f)\in \CS(Y)=\CS(F^2)$, we can check that the following formulas hold:
\begin{align}
(\omega(n(b))\phi)(x,y)&=\psi(bxy)\phi(x,y), b\in F,\label{eq41}\\
(\omega(t(a_1,a_2))\phi)(x,y)&=|a_1a_2^{-1}|^{1/2}\phi(a_1x,a_2^{-1}y), a_1, a_2\in F^\times, \label{eq42}\\
(\omega(w)\phi)(x,y)&=\int_{F\times F}\psi(xv-yu)\phi(u,v)dudv. \label{eq43}
\end{align}

For a quasi-character $\chi$ of $F^\times$, we define 
\begin{equation}\label{eq44}\phi_\chi(x,y)=\int_{F^\times}\chi^{-1}(a)\phi(ax,a^{-1}y)da.\end{equation} Then $(\omega(z)\phi)_\chi=\chi(z)\phi_\chi,$ for $z\in Z\cong F^\times$. The projection map $\omega_\psi\ra \omega_{\psi,\chi}$ can be identified with
$$(\omega, \CS(Y))\ra \Ind_B^H(1\otimes \chi)$$
$$\phi\mapsto (\omega(h)\phi)_\chi(0,1).$$

\begin{lem}{\label{lem43}}
Let $f\in \Ind_B^H(1\otimes \chi)$ be the function defined by $f(h)=(\omega (h)\phi)_\chi(0,1)$, for some $\phi\in \CS(Y)=\CS(F^2)$. We define $\lambda(f)=\phi_\chi(1,1)$. Then $\lambda$ is well-defined and defines a nontrivial Whittaker functional on $\omega_{\psi,\chi}$.
\end{lem}
\begin{proof}
The kernel of the map $\omega_\psi \ra \omega_{\psi,\chi}$ consists functions of the form $\omega(z)\phi-\chi(z)\phi, z\in Z, \phi \in \CS(Y)$. To show $\lambda$ is well-defined, it suffices to show that $ \lambda$ vanishes on the kernel, i.e., 
$$(\omega(z)\phi)_\chi(1,1)=(\chi(z)\phi)_\chi(1,1).$$
This is clear. For $n=n(b)\in N$, one preimage of $r(n)f$ is $\omega(n)\phi$. Thus 
$$\lambda(r(n)f)=(\omega(n)\phi)_\chi(1,1)=\psi(b)\phi_\chi(1,1),$$
by Eq.(\ref{eq41}) and Eq.(\ref{eq44}). This shows that $\lambda$ is a Whittaker functional.
\end{proof}

\subsection{The Local Zeta Integral}
Recall that $T$ is the diagonal subgroup of $H=\GL_2(F)$ and we have $T\cong F^\times \times F^\times$. A quasi-character $\eta$ of $T$ consists a pair of characters $(\eta_1,\eta_2)$ of $F^\times$: $\eta(t(a,b))=\eta_1(a)\eta_2(b)$.
 Let $||~||$ be the character of $T$ defined by 
$$||t(a,b)||=|ab^{-1}|_F.$$
Let $\eta=(\eta_1, \eta_2)$ be a quasi-character of $T$. We consider the induced representation $\Ind_B^H(\eta|~|^{s-1/2})$ of $H$.

Let $(\pi,V)$ be an infinite-dimensional irreducible admissible representation of $\GL_2(F)$ with central character $\pi$, let $\eta=(\eta_1,\eta_2)$ be a quasi-character of $T$ and $ \chi$ be quasi-characters of $F^\times$. We require that
$$\omega_\pi\cdot \chi \cdot \eta_1 \eta_2=1.$$

 For $W\in \CW(\pi,\psi)$, $\theta\in \CW(\omega( \psi^{-1},\chi), \psi^{-1})$, and $f_s\in \Ind_B^H(\eta|~|^{s-1/2})$, similar to the $\RU(1,1)$ case, we consider the following local zeta integra
$$\Psi(W,\theta,f_s)=\int_{R\setminus H}W(h)\theta(h)f_s(h)dh.$$
\textbf{Remark:} This is the local zeta integral at the split place of the global zeta integral for $\RU(1,1)$ defined in \cite{GRS}, which is also a special case of the Rankin-Selberg type local zeta integral of $\GL_2\times \GL_2$ as defined by Jacquet in \cite{J}. For the Rankin-Selberg integral for general $\GL_n\times \GL_m$, see \cite{JPSS}.

For $\Phi\in \CS(F^2)$,  we define
\begin{equation}{\label{eq45}}f(h,s,\Phi,\eta)=\eta_1(\det(h))|\det(h)|_F^s\int_{F^\times}\Phi((0,r)h)\eta_1(r)\eta_2^{-1}(r)|r|_F^{2s}d^*r.\end{equation}
The integral in the definition of $f(s,h,\Phi,\eta)$ converges for $\Re(s)$ large enough and defines a meromorphic function of $s$.
It is easy to check that $f(s,\cdot,\Phi,h)\in \Ind_B^H(\eta|~|^{s-1/2})$. For $\Re(s)$ large enough, every element in $\Ind_B^H(\eta|~|^{s-1/2})$ is of the form $f(h,s, \Phi, \eta)$ for some $\Phi\in \CS(F^2)$, see \cite{JL} for example.
We denote
$$\Psi(s,W,\theta,\Phi, \eta_1)=\Psi(s,W,\theta,f(~, s,\Phi,\eta)).$$
The above notation makes sense because $\eta_2$, and hence $\eta$ is uniquely determined by the other parameters.
It is standard to show that $\Psi(s,W,\theta,\Phi,\eta_1)$ converges absolutely for $\Re(s)>>0$ and defines a meromorphic function of $s$, see \cite{J}.

Plugin the definition of $f(h,s,\Phi, \eta)$, it is easy to see that
$$\Psi(s,W,\theta,\Phi,\eta_1)=\int_{N\setminus H} W(h)\theta(h)\eta_1(\det(h))|\det(h)|^s \Phi((0,1)h)dh.$$
Note that in this expression, $\eta_2$ is hidden.

Next, we suppose that the Whittaker function $\theta\in \CW(\omega_{\psi^{-1},\chi},\psi^{-1})$ is associated with $\phi\in \CS(F^2)$, i.e., we fix a nonzero $\lambda\in \Hom_N(\omega_{\psi^{-1},\chi},\psi^{-1})$, then $\theta(h)=\lambda( \omega(h)\phi )$. By Lemma \ref{lem43}, we have
\begin{equation}\label{eq46}\theta(h)=\int_{F^\times}\chi^{-1}(a)(\omega(h)\phi)(a,a^{-1})d^*a=\int_{Z}\chi^{-1}(z)(\omega(zh)\phi)(1,1)dz.\end{equation}
We will use this expression later.

\subsection{The local $\gamma$-factor and local functional equation}
We recall the local functional equation of the Rankin-Selberg zeta integral for $\GL_2$, see \cite{J}, Theorem 14.7.

For an irreducible smooth representation $\pi$ of $\GL_2(F)$, we have an isomorphism $\tilde \pi \cong \pi\otimes \omega_\pi^{-1}$. For $W\in \CW(\pi,\psi)$, if we define
$$\widetilde W(h)=W(h)\omega_\pi^{-1}(\det(h)),$$
then $\widetilde W\in \CW(\tilde \pi, \psi)$. For $\Phi\in \CS(F^2)$, we define the Fourier transform $\hat \Phi$ by 
$$\hat \Phi(x,y)=\int \Phi(u,v)\psi(uy-vx)dudv.$$
\begin{thm}[Jacquet, Theorem 14.7 of \cite{J}]
There is a meromorphic function $\gamma(s,\pi,\omega_{\psi^{-1},\chi}, \eta_1)$ such that
$$\Psi(1-s,\widetilde W, \tilde \theta, \hat \Phi, \eta_1^{-1})=\gamma(s,\pi,\omega_{\psi^{-1},\chi}, \eta_1)\Psi(s,W,\theta, \Phi, \eta_1).$$
\end{thm}
It is clear that $\Psi(1-s,\widetilde W, \tilde \theta, \hat \Phi, \eta_1^{-1})=\Psi(1-s, W,\theta, \hat \Phi, \eta_2)$, where $\eta_2=\omega_\pi^{-1}\omega_\theta^{-1}\eta_1^{-1}$.
\begin{thm}[Jacquet, Theorem 15.1 of \cite{J}, Multiplicativity of $\gamma$-factors]\label{thm44}
If $\chi\ne |~|^{\pm 1}$ so that $\omega_{\psi,\chi}\cong \Ind_B^H(1\otimes \chi)$ is irreducible, then $$\gamma(s,\pi, \omega_{\psi^{-1},\chi}, \eta)=\gamma(s,\pi, \eta_1)\gamma(s,\pi, \chi \eta_1),$$
where $\gamma(s,\pi, \eta_1)=\gamma(s,\pi\otimes \eta_1)$ is the $\gamma$-factor of $\pi\otimes \eta_1$ defined in \cite{JL}.
\end{thm}

\section{A new proof of the local converse theorem for $\GL(2)$}
\subsection{Howe vector for GL(2,F)}
In this section, we consider the Howe vectors for representations of $\GL_2(F)$. The theory is parallel to $\S$3.1. We omit some details, which can be found in \cite{Ba2}.

Let $p=p_F$ be a prime element of $F$. Let $\CP=p_F\CO_F$ be the maximal ideal of $F$. Let $\psi$ be an additive character on $F$ such that $\psi(\CO_F)=1$ and $\psi(p^{-1}\CO)\ne 1$.

Let $K_m\subset \GL_2(F)$ be a congruence subgroup defined by $1+M_2(P^m)$.  Set
$$d_m=\begin{pmatrix} p^{-m}&\\ &  p^{m}\end{pmatrix}.$$
Let $J_m=d_mK_md_m^{-1}$. Explictly, $$J_m=\begin{pmatrix} 1+P^m& P^{-m}\\ P^{3m}& 1+P^m\end{pmatrix}\cap H.$$

Let $\tau_m$ be the character of $K_m$ defined by
$$\tau_m(k)=\psi(p^{-2m}k_{1,2}), k=(k_{i,l})\in K_m.$$
By our assumption on $\psi$, it is easy to see that $\tau_m$ is indeed a character on $K_m$. Define $\psi_m$ on $J_m$ by
$$\psi_m(j)=\tau_m(d_m^{-1}jd_m), j\in J_m.$$
One can check that $\psi_m$ and $\psi$ agree on $N_m:=J_m\cap N$.

Let $(\pi,V)$ be a $\psi$ generic representation. We fix a Whittaker functional. Let $v\in V$ be such that $W_v(1)=1$. Let $N_m=J_m\cap N$. For $m\ge 1$, define
$$v_m=\frac{1}{\Vol(N_m)}\int_{N_m}\psi(n)^{-1}\pi(n)vdn.$$
Let $L$ be an integer such that $v$ is fixed by $K_L$.

\begin{lem}{\label {lem51}} We have
\begin{enumerate}
\item $W_{v_m}(1)=1.$
\item If $m\ge L$, $\pi(j)v_m=\psi_m(j)v_m$ for all $j\in J_m$.
\item If $k\le m$, then
$$v_m=\frac{1}{\Vol(N_m)}\int_{N_m}\psi(n)^{-1}\pi(n)v_kdn.$$
\end{enumerate}
\end{lem}
The vectors $v_m$ are called Howe vectors.

Let $w=\begin{pmatrix} &1\\ -1& \end{pmatrix}$ be the nontrivial Weyl element of $H$. Recall that we use $t(a)$ to denote the element $\diag(a,1)$ for $a\in F^\times$. \begin{lem}{\label{lem52}}
Let $v_m,m\ge L$, be Howe vectors defined as above, and let $W_{v_m}$ be the Whittaker functions associated to $v_m$. Then
\begin{enumerate}
\item $W_{v_m}(t(a))\ne 0 \textrm{ implies } a\in 1+\CP_F^m.$
\item  $W_{v_m}(t(a)w)\ne 0$ implies $a\in \CP_F^{-3m}$.
\end{enumerate}
\end{lem}
\begin{proof}
(1) For $x\in \CP_F^{-m}$, we have $n(x)\in N_m$. From Lemma \ref{lem51} and the relation
$$t(a)n(x)=n(ax)t(a),$$
we get
$$\psi_m(x)W_{v_m}(m(a))=\psi(ax)W_{v_m}(m(a)).$$
If $W_{v_m}(m(a))\ne 0$, we have $(1- a)x\in \Ker(\psi)$ for any $x\in \CP^{-m}$. Thus $a\in 1+\CP^m$.  

(2) For $x\in \CP_F^{3m}$, we have $\bar n(-x)\in \bar N_m:=J_m\cap \bar N$. It is easy to check the relation
$$t(a)w\bar n(-x)=n(a x)t(a)w.$$
By Lemma \ref{lem51}, we have
$$W_{v_m}(t(a)w)=\psi(a x)W_{v_m}(t(a)w).$$
Thus $W_{v_m}(m(a)w)\ne 0$ implies $\psi(a x)=1$ for all $x\in \CP_F^{3m}$, i.e., $a\in \CP_F^{-3m}$.
\end{proof}

\begin{prop}{\label{prop53}}
Let $(\pi,V)$ and $(\pi',V')$ be two $\psi$-generic representation of $H$ with the same central character. Choose $v\in V,v'\in V'$ so that $W_v(1)=1=W_{v'}(1)$ and define $v_m\in V,v_m'\in V'$ as above. Let $L$ be an integer such that $v,v'$ are fixed by $K_L$. Let $m\ge 3L$ and $n_0\in N$. Then
\begin{enumerate}
\item $W_{v_m}(t(a))=1=W_{v'_m}(t(a))$ for all $a\in 1+\CP^m$.
\item $W_{v_m}(b)=W_{v_m'}(b)$ for all $b\in B$.
\item If $n_0\in N_m$, then $W_{v_m}(twn_0)=\psi(n_0)W_{v_m}(tw)$ for all $t\in T=M$.
\item If $n_0\notin N_m$, then $W_{v_m}(twn_0)=W_{v'_m}(twn_0)$ for all $t\in T$.
\end{enumerate}
\end{prop}
\begin{proof}
(1) For $a\in 1+P^m,$ we have $t(a)\in J_m$. Thus $W_{v_m}(t(a))=\psi_m(t(a))W_{v_m}(1)=1$.

(2) Since $B=NT$, it suffices to check that $W_{v_m}(t)=W_{v_m'}(t)$ for all $t\in T$. From (1) and the above lemma, we have $W_{v_m}(t(a))=W_{v'_m}(t(a))$ for all $a\in F^\times$. For $t=\diag(a,b)=\diag(ab^{-1},1)\diag(b,b)$, thus 
$$W_{v_m}(t)=\omega_{\pi}(b)W_{v_m}(ab^{-1})=\omega_{\pi'}(b)W_{v_m'}(t(ab^{-1}))=W_{v_m'}(t).$$ 

(3) For $n_0\in N_m$, we have $\pi(n)v_m=\psi(n_0)v_m$. The assertion is clear.

(4) We have
$$W_{v_m}(twn_0)=\Vol(N_m)^{-1}\int_{N_m}W_{v_L}(twn_0n)\psi^{-1}(n)dn.$$
Let $n'=n_0n$. Since $n_0\notin N_m,n\in N_m$, we get $n'\notin N_m$.

 Suppose $n'=n(x)$. Then $n\notin N_m$ is equivalent to $x\notin P^{-m}$. In particular, $x\ne 0$.
We have
$$\begin{pmatrix} &1 \\ -1& \end{pmatrix}\begin{pmatrix} 1&x \\ &1 \end{pmatrix}=\begin{pmatrix}- x^{-1}& 1\\ &-x \end{pmatrix}\begin{pmatrix} 1& \\ x^{-1} &1\end{pmatrix}$$
Let $b=\begin{pmatrix}- x^{-1}& 1\\ &-x \end{pmatrix},j=\begin{pmatrix} 1& \\ x^{-1} &1\end{pmatrix}$. Then $b\in B_0, j\in J_L$ by assumption. Thus
$$W_{v_L}(twn')=W_{v_L}(tbj)=W_{v_L}(tb).$$
By (1), we have $W_{v_L}(tb)=W_{v_L'}(tb)$. Now it is clear that $W_{v_m}(twn_0)=W_{v'_m}(twn_0)$.
\end{proof}
\subsection{Howe vectors for Weil representations}\label{sec52}
We first quote two lemmas on Fourier analysis on $p$-adic fields. Before that, we recall some notations. Let $\chi$ be a character of $F^\times$. If $\chi$ is unramified, we denote $\deg(\chi)=0$. If $\chi$ us ramified, we denote $\deg(\chi)$ the integer $m$ such that $\chi|_{1+\CP^m}=1$ but $\chi|_{1+\CP^{m-1}}\ne 1$. Recall that we fixed an unramified additive character $\psi$ of $F$. In the integrals of the following lemmas, we use the standard Haar measure, i.e., we fix a Haar measure $dx$ on $F$ such that $\Vol(\CO_F, dx)=1$, and we denote $d^* x=\frac{dx}{|x|}$ the multiplicative Haar measure on $F^\times$. Note that we have $\Vol(\CO_F^\times, d^*x)=1-q_F^{-1}$ and $\vol(1+\CP_F^m,d^*x)=q_F^{-m}$.

\begin{lem}{\label{lem54}}
\begin{enumerate}
\item We have
$$\int_{\CO_F^\times}\psi(p^ku)d^*u=\left\{\begin{array}{lll}1-q_F^{-1}=\Vol(\CO_F^\times, d^*u):=c_0, & k\ge 0 \\ 
 q_F^{-1}(\sum_{a\in \CO^\times/(1+\CP_F)}\psi(p^{-1}a))=-q_F^{-1}:=c_1, & k=-1,\\
0,& k\le -2.\end{array}\right.$$
\item Let $\chi$ be a ramified character of $F^\times$ of degree $h$, $h\ge 1$. Then
$$\int_{\CO_F^\times}\chi(u)\psi(p^ku)d^*u=\left\{\begin{array}{lll}0 & k\ne -h \\ 
\ne 0, & k=-h.\\\end{array}\right.$$
\end{enumerate}
\end{lem}
\begin{proof}
(1) is easy and can be found in Page 21 of \cite{Ta}. (2) is Lemma 5.1 of \cite{Ta}, Page 49.
\end{proof}

Let $k$ be a positive integer, $\chi$ be a character of $F^\times$ and $a\in F^\times$, we consider the following integral
$$F_\chi(k,a)=\int_{|x|=q^k}\psi(x+ax^{-1})\chi(x)d^*x.$$
\begin{lem}{\label{lem55}}
Suppose that $|a|=q_F^n$ with $1\le k<n$.
\begin{enumerate}
\item If $\chi$ is unramified, then $F_\chi(k,a)\ne 0$ if and only if $n$ is even and $k=n/2$.
\item If $\chi$ is ramified of degree $h\ge 1$, $F_\chi(k,a)\ne 0$ if and only if one of the following holds:
\begin{enumerate}
\item[(a)] $n$ is even, $n\ge 2h$ and $k=n/2$,
\item [(b)] $n$ is even, $h<n<2h$ and $k=n/2$, $k=h$ or $k=n-h$,
\item[(c)] $n$ is odd, $h<n<2h$ and $k=h$ or $k=n-h$.
\end{enumerate}
\end{enumerate}
\end{lem}
\begin{proof} This is Lemma 5.23 of \cite{Ta}, page 67.
\end{proof}

For a positive integer $m$, we consider the function $\phi^m\in \CS(F^2)$ defined by $$\phi^m(x,y)=\Upsilon_{1+\CP^m}(x)\cdot \Upsilon_{1+\CP_F^m}(y),$$
where for a subset $A\subset F$, $\Upsilon_A(x)$ denote the characteristic function of $A$. 
\begin{prop}{\label{prop56}}
\begin{enumerate}
\item For $n\in N_m$, we have
$$\omega_{\psi^{-1}}(n)\phi^{m}=\psi^{-1}(n)\phi^{m}.$$
\item For $\bar n \in \bar N_{m}=\bar N\cap J_m$, we have $$\omega_{\psi^{-1}}(\bar n)\phi^{m,\chi}=\phi^{m,\chi}.$$
\end{enumerate}
\end{prop}
\begin{proof}
(1) For $n=n(b)\in N_m$, we have
\begin{align*}
\quad\omega_{\psi^{-1}}(n(b))\phi^{m}(x,y)
=\psi^{-1}(bxy))\phi^{m}(x,y).
\end{align*}
For $(x,y)\in \supp(\phi^{m,\chi})$, i.e., $x,y\in 1+\CP_F^m$ and $n(b)\in N_m$, i.e., $b\in \CP_F^{-m}$, we have $bxy -b\in \CO_F$. Since $\psi$ is unramified, we get $\psi^{-1}(bxy)=\psi(b)$.
Thus $$\omega_{\psi^{-1}}(n(b))\phi^{m}=\psi^{-1}(b)\phi^{m}.$$

(2) 
For $\bar n\in \bar N_m$, we can write $\bar n= w^{-1}n(b)w$ with $b\in \CP_F^{3m}$.
Let $\phi'=\omega(w)\phi^{m}$. We have
\begin{align*}
\phi'(x,y)&=\int_{F^2}\phi(u,v)\psi^{-1}(xv-yu)dudv\\
&=\int_{1+\CP_F^m}\psi^{-1}(xv)dv \int_{1+\CP^m}\psi(yu)du\\
&=\psi^{-1}(x) \Upsilon_{\CP_F^{-m}}(x)\cdot \psi(y)\Upsilon_{\CP_F^{-m}}(y).
\end{align*}
For $(x,y)\in \supp \phi'=\CP_F^{-m}\times \CP_F^{-m}$, we have $bxy\in \CO_F$, and thus
\begin{align*}
\omega(n(b))\phi'(x,y)&=\psi^{-1}(b xy)\phi'(x)=\phi'(x,y),
\end{align*}
i.e., $\omega(n(b))$ fixes $\phi'$. Then
$$\omega(\bar n)\phi^{m}=\omega(w^{-1})\omega(n(b))\omega(w)\phi^{m}$$
$$=\omega(w^{-1}) \omega(n(b)) \phi'=\omega(w^{-1}) \phi'=\omega(w^{-1})\omega(w)\phi^{m}=\phi^{m}.$$
This completes the proof.
\end{proof}
Recall that, for a character $\chi$ of $F^\times$, we have a projection map $\omega_{\psi^{-1}}\ra \omega_{\psi^{-1}, \chi}$, defined by $\phi\mapsto \phi_\chi$. The following corollary follows from Proposition \ref{prop56} directly. 
\begin{cor}{\label{cor57}}
\begin{enumerate}
\item For $n\in N_m$, we have
$$\omega_{\psi^{-1},\chi}(n)\phi_\chi^{m}=\psi^{-1}(n)\phi_\chi^{m}.$$
\item  For $\bar n \in \bar N_{m}=\bar N\cap J_m$, we have $$\omega_{\psi^{-1},\chi}(\bar n)\phi^{m}_\chi=\phi^{m}_\chi.$$
\end{enumerate}
\end{cor}

By Lemma \ref{lem43}, the Whittaker function $\theta^{m,\chi}$ associated with $\phi^m_\chi$ for the representation $\omega_{\psi^{-1},\chi}$ is 
\begin{equation}\label{eq51}\theta^{m,\chi}(h)=(\omega_{\psi^{-1}}(h)\phi^m)_\chi(1,1)=\int_{F^\times}\chi^{-1}(u)(\omega_{\psi^{-1}}(h)\phi^m)(u,u^{-1})d^*u.\end{equation}
\begin{prop}\label{prop58}
Given positive integers $N, m$, with $m\ge N$. Then there exist a finite number of characters $\wpair{\chi_i}$ of $F^\times$ satisfying the following conditions:
\begin{enumerate}
\item $\chi_i\ne |~|^{\pm 1}$, 
\item $\deg(\chi_i)\le N$, and
\item for all $a\in F^\times$ with $|a|\le q_F^N$, there exists a character $\chi_i$ such that
$$\theta^{m,\chi_i}(t(a)w)\ne 0.$$
\end{enumerate}
\end{prop}
Recall that $w=\begin{pmatrix}&1\\ -1& \end{pmatrix}$. This proposition is the key to prove the local converse theorem.
\begin{proof}
By the Whittaker function formula of Weil representation, Eq.(\ref{eq46}) or Eq.(\ref{eq51}), and the Weil representation formulas, Eq.(\ref{eq42}) and Eq.(\ref{eq43}), we have 
\begin{align*}
&\quad \theta^{m,\chi}(t(a)w)\\
&=\int_{F^\times}\chi^{-1}(u)(\omega_{\psi^{-1}}(t(a)w)\phi^m)(u,u^{-1})d^*u\\
&=|a|^{1/2}\int_{F^\times}\chi^{-1}(u)\omega(w)\phi^m(au,u^{-1})d^*u\\
&=|a|^{1/2}\int_{F^\times}\chi^{-1}(u)\int_{F\times F}\psi^{-1}(auy-u^{-1}x)\phi^m(x,y)dxdyd^*u\\
&=|a|^{1/2}\int_{F^\times}\chi^{-1}(u)\left(\int_{1+\CP_F^m} \psi^{-1}(-u^{-1}x)dx \int_{1+\CP_F^m}\psi^{-1}(auy)dy\right) d^*u.
\end{align*}
 It is easy to see that the inner integral 
 $$ \int_{1+\CP_F^m} \psi^{-1}(-u^{-1}x)dx \int_{1+\CP_F^m}\psi^{-1}(auy)dy$$
is $\psi(u^{-1}-ua)$ if $ u\in a^{-1}\CP^{-m},$ and $ u^{-1}\in \CP^{-m}$ and zero otherwise. Suppose that $|a|=q^n$, with $n\le N$. We then have
\begin{align*}\theta^{m,\chi}(t(a)w)&=q_F^{n/2}\int_{q^{-m}\le |u|\le q^{m-n}}\chi^{-1}(u)\psi(u^{-1}-au)d^*u
\end{align*}
Note that $2m\ge N\ge n$, thus the domain $\wpair{u\in F^\times| q^{-m}\le|u|\le q^{m-n}}$ is nonempty. We first assume that $n\le 0$, i.e., $|a|\le 1$. Let $\chi_0$ be an unramified character. By Lemma \ref{lem54}, we have
\begin{align*}
q_F^{-n/2}\theta^{m,\chi_0}(t(a)w)&=\int_{q^{-m}\le |u|\le 1}\chi_0^{-1}(u)\psi(u^{-1})du+\int_{1<|u|\le q^{m-n}}\chi_0^{-1}(u)\psi(-ua)d^*u\\
&=\sum_{k=-m}^0 \chi_0(p^{k})\int_{\CO_F^\times} \psi(p^ku_0^{-1})d^*u_0+\sum_{k=1}^{m-n}\chi_0(p^k)\int_{\CO_F^\times}\psi(-p^{-k-n}u_0)d^*u_0\\
&= \chi_0(p^{-1})c_1+c_0+\chi_0(p^{1-n})c_1+\sum_{k=1}^{-n}\chi_0(p^k)c_0, 
\end{align*}
where the sum $ \sum_{k=1}^{-n}\chi_0(p^k)$ should be interpreted as 0 if $n=0$. Since the character $\chi_0$ is determined by $\chi_0(p)$, it is easy to choose $\chi_0(p)$ and hence $\chi_0$ which might depend on $|a|$  such that $\chi_0\ne |~|^{\pm 1}$ and $$\theta^{m,\chi_0}(t(a)w)\ne 0, \textrm{ for } |a|\le 1. $$ For example, one can take $\chi_0$ to be the trivial character if $q_F\ne 2, 3$. If $q_F=2$ and $n\ne -1$, or $q_F=3, n\ne 0$, one can still take $\chi_0$ to be the trivial character. If $q_F=2$ and $n=-1$ or $q_F=3$ and $n=0$, one can take $\chi_0=|~|^{-2}$. From this discussion, we know that at most two characters will work for all $a$ with $|a|\le 1$.

Now suppose that $n\ge 1$. We have
\begin{align}
&|a|^{-1/2}\theta^{\chi}_{\phi_m}(t(a)w)\label{eq52}\\
&=\int_{q^{-m}\le |u|\le q^{-n}}\chi^{-1}(u)\psi(u^{-1})d^*u \label{eq53}\\
&+\int_{ q^{-n}< |u|<1}\chi^{-1}(u)\psi(u^{-1}-au)d^*u \label{eq54}\\
&+\int_{1\le |u|\le q^{m-n} }\chi^{-1}(u)\psi(-au)d^*u. \label{eq55}
\end{align}
If $n=1$, the term Eq.(\ref{eq54}) is zero. Let $\chi=\chi_0$ be an unramified character, by Lemma \ref{lem54}, the term Eq.(\ref{eq53})
becomes $\chi_0(p)c_1$, and the term Eq. (\ref{eq55}) becomes $c_1$. Thus
$$q_F^{-1/2}\theta^{m,\chi}(t(a)w)=(\chi_0(p)+1 )c_1, \textrm{ if } |a|=q_F.$$
Thus we can choose a single $\chi_0\ne |~|^{\pm1}$ such that $$\theta^{m,\chi_0}(t(a)w)\ne 0, \forall a \textrm{ with } |a|= q_F.$$

Next we consider the case $n\ge 2$ and $n$ is even. We still take an unramified character $\chi_0\ne |~|^{\pm1}$. Then by Lemma \ref{lem54}, the term Eq.(\ref{eq53}) and the term Eq.(\ref{eq55}) vanish.
By Lemma \ref{lem55} and Eq.(\ref{eq52}), we have
\begin{align*}
&\quad q_F^{-n/2}\theta^{m,\chi}(t(a)w)\\
&= \int_{q^{-n}<|u|<1}\chi_0^{-1}(u)\psi(u^{-1}-au)d^*u\\
&=\sum_{k=1}^{n-1}\int_{|u|=q^k}\chi_0(u)\psi(u-\frac{a}{u})d^*u\\
&=\sum_{k=1}^{n-1}F_{\chi_0}(k,-a)\\
&=F_{\chi_0}(n/2, -a)\ne 0.
\end{align*}

Next we suppose that $|a|=q^n$ with $n$ odd and $n\ge 3$. Let $\chi$ be a ramified character of degree $n$. By Lemma \ref{lem55}, the term Eq.(\ref{eq54}) vanishes. In fact, we have
\begin{align*}
& \int_{q^{-n}<|u|<1}\chi^{-1}(u)\psi(u^{-1}-au)d^*u\\
&=\sum_{k=1}^{n-1}\int_{|u|=q^k}\chi(u)\psi(u-\frac{a}{u})d^*u\\
&=\sum_{k=1}^{n-1}F_{\chi}(k,-a).
\end{align*}
By Lemma \ref{lem55}, each term $F_\chi(k,-a)$ vanishes.  By Lemma \ref{lem54}, the term Eq.(\ref{eq53}) is 
\begin{align*}
&\sum_{k=-m}^{-n}\chi(p)^k\int_{\CO_F^\times}\chi(u_0)\psi(p^ku_0)d^*u_0\\
=&\chi(p)^{-n}\int_{\CO_F^\times}\chi(u_0)\psi(p^{-n}u_0) d^*u_0 \ne 0.
\end{align*}

If we suppose $a=p^{-n}a_0$ with $a_0\in \CO_F^\times$, then the term Eq (\ref{eq55}) is 
\begin{align*}
&\sum_{k=0}^{m-n}\chi(p)^k\int_{\CO_F^\times}\chi^{-1}(u_0)\psi(-a_0u_0p^{-k-n})d^*u_0\\
=&\sum_{k=0}^{m-n}\chi(p)^k \chi(a_0)\int_{\CO_F^\times} \chi^{-1}(u_0)\psi(-p^{-k-n}u_0)d^*u_0\\
=&\chi(-a_0)\int_{\CO_F^\times}\chi^{-1}(u_0)\psi(p^{-n}u_0)d^*u_0.
\end{align*}
Thus we have
\begin{align}\label{eq56}
q_F^{-n/2}\theta^{m,\chi}(t(a)w)&=\chi^{-1}(p^n)\int_{\CO_F^\times}\chi(u)\psi(p^{-n}u)d^*u+\chi(-a_0)\int_{\CO_F^\times}\chi^{-1}(u)\psi(p^{-n}u)d^*u.
\end{align}
For a character $\chi$ of degree $n$, write $c(\chi)=\int_{\CO_F^\times}\chi(u)\psi(p^{-n}u)d^*u\ne 0$ temporarily, then Eq.(\ref{eq56}) can be written as
\begin{equation}\label{eq57} q_F^{-n/2}\theta^{m,\chi}(t(a)w)= \chi(p)^{-n}c(\chi)+\chi(-a_0)c(\chi^{-1}).\end{equation}

We claim that, for $n\ge 2$ with $n$ odd, we can find a character $\chi$ such that $\theta^{m,\chi}(t(a)w)\ne 0$. In fact, if exists a character $\chi_0$ of degree $n$ such that $\theta^{m,\chi_0}(t(a)w)=0$, which is equivalent to $\chi_0(-a_0)c(\chi_0^{-1})=-\chi_0(p)^{-n}c(\chi_0)$ by Eq.(\ref{eq57}). Let $\chi_s$ be the character of $F^\times$ defined by $\chi_0|~|^s$. Note that $\deg( \chi_s)=n$, $\chi_s(-a_0)=\chi_0(a_0)$ and $c(\chi_s)=c(\chi_0)$. Then 
\begin{align*}q_F^{-n/2}\theta^{m,\chi_s}(t(a)w)&=q_F^{ns}\chi_0^{-1}(p^n)c(\chi_0)+\chi_0(-a_0)c(\chi_0^{-1})\\
&=(q^{ns}-1)\chi_0^{-1}(p^n)c(\chi_0).
 \end{align*}
If $s\ne 0$, i.e., $q^{ns}\ne 1$, then $$\theta^{m,\chi_s}(t(a)w)\ne 0.$$ 
This shows that for any $a$ with $a\in \CP_F^{-N}$, we can find a character $\chi$ such that $\theta^{m,\chi}(t(a)w)\ne 0$. We also need to prove the ``finiteness" part. By the above discussion, it suffices to show that for each odd $n$ with $n\ge 2$, there exists a finite number of characters $\wpair{\chi_{i,n}}$ such that for each $a$ with $|a|=q_F^n$, we can find a character $\chi_{i,n}$ such that $\theta^{m,\chi_{i,n}}(t(a)w)\ne 0$.  

In fact, for a character $\chi$ of degree $n$, since it is trivial on $1+\CP^n$, for any $u\in a_0(1+\CP^n)$, Eq.(\ref{eq56}) shows that
$$\theta^{m,\chi}(t(p^{-n}u)w)=\theta^{m,\chi}(t(p^{-n}a_0)w).$$
This means that if a character $\chi$ works for $p^{-n}a_0$, then it works for $p^{-n}u$ for all $u\in a_0(1+\CP^n)$.
Since $\CO_F^\times/(1+\CP^n)$ is finite, the assertion follows.
\end{proof}

\begin{cor}{\label{cor59}}
Given positive integers $m,N$ with $m\ge N$. For any character $\chi$ of $F^\times$, there is an integer $d=d(N,\chi)$ such that 
$$\theta^{m,\chi}(t(a_0a)w)=\theta^{m,\chi}(t(a)w), \forall a_0\in 1+\CP_F^d, a\in \CP_F^{-N}.$$
In fact, we can take $d=\max\wpair{\deg(\chi),N}$.
\end{cor}
\begin{proof}
The function $\theta^{m,\chi}(t(a)w)$ with variable $a$ is continuous, thus there is an integer $d$ such that $\theta^{m,\chi}(t(a_0a)w)=\theta^{m,\chi}(t(a)w)$ for all $a\in 1+\CP_F^d$. We want to show that such an integer $d$ can be chosen independent of $m$.

From the proof of Proposition {\ref{prop58}}, the term $|a|^{-1/2}\theta^{m,\chi}(t(a)w)$ have the expression Eq.(\ref{eq52}). If $n\le 0$, i.e., $q^{-n}\ge 1$, we should view the term Eq.(\ref{eq54}) as zero. Denote the term Eq.(\ref{eq53}) (resp. Eq.(\ref{eq54}), Eq.(\ref{eq55})) by $G_1(a)$ (resp. $G_2(a)$, $G_3(a)$). The term $G_1(a)$ only depends on $|a|$, and thus $G_1(a_0a)=G_1(a)$ for all $a_0\in \CO_F^\times$. For $u$ with $q_F^{-n}<|u|<1$, we have $p^{N}au\in \CO_F$. From this, one can check that $G_2(a_0a)=G_2(a)$ for all $a_0\in 1+\CP_F^N$ easily. 

We have $$G_3(a)=\int_{1\le |u|\le q^{m-n} }\chi^{-1}(u)\psi(-au)d^*u.$$ 
Thus for $a_0\in \CO_F^\times$, by changing variable, we have
\begin{align*}
G_3(a_0a)&=\int_{1\le |u|\le q^{m-n} }\chi^{-1}(u)\psi(-a_0au)d^*u\\
&=\chi(a_0)\int_{1\le |u|\le q^{m-n} }\chi^{-1}(u)\psi(-au)d^*u\\
&=\chi(a_0)G_3(a).
\end{align*}
Thus we have $G_3(a_0a)=G_3(a)$ for $a\in 1+\CP_F^{\deg(\chi)}$.
\end{proof}

\subsection{A new proof of the local converse theorem for $\GL_2$} Recall the local converse theorem for $\GL_2$:
\begin{thm}[Corollary 2.19 of \cite{JL}]\label{thm510}
Let $\pi,\pi'$ be two infinite dimensional irreducible smooth representation of $\GL_2(F)$ with the same central character. If $\gamma(s,\pi,\eta)=\gamma(s,\pi',\eta)$ for all quasi-characters $\eta$ of $F^\times$, then $\pi\cong \pi'$.
\end{thm}
We are going to give a new proof of this theorem. In fact, we will prove the following

\begin{thm}{\label{thm511}}
Let $\pi,\pi'$ be two $\psi$-generic irreducible smooth infinite dimensional representations of $\GL_2(F)$ with the same central character $\omega_\pi=\omega_{\pi'}$.
\begin{enumerate}
\item If $\gamma(s,\pi, \omega_{\psi^{-1}, \chi},\eta_1)=\gamma(s,\pi', \omega_{\psi^{-1},\chi},\eta_1)$ for all quasi-characters $\chi,\eta_1,$ of $F^\times$ with $\chi\ne |~|^{\pm1}$, 
then $\pi$ is isomorphic to $\pi'$.

\item For a fixed quasi-character $\chi$ of $F^\times$, there is an integer $l_0=l_0(\pi,\pi',\chi)$ such that for all quasi-characters $\eta_1,$ of $F^\times$ with 
 $\deg(\eta_1)>l_0$, we have
$$\gamma(s,\pi, \omega_{\psi^{-1}, \chi},\eta)=\gamma(s,\pi', \omega_{\psi^{-1},\chi},\eta). $$
\end{enumerate}
\end{thm}
\begin{rem}\upshape
(1) Note that, the condition of (1) of Theorem \ref{thm511}  is weaker than that of Theorem \ref{thm510} by the multiplicativity theorem of $\gamma$-factors, Theorem \ref{thm44}. Thus it is clear that part (1) of Theorem \ref{thm511} implies Theorem \ref{thm510}. \\
(2) The conclusion of Part (2) of Theorem \ref{thm511} is weaker than the stability of gamma factors for $\GL_2$.

\end{rem}

Before the proof, we recall that in $\S$\ref{sec33}, we defined a function $\Phi_{i,l}\in \CS(F^2)$. We need to compute $f(s,h,\Phi_{i,l}, \eta)$ in our case. The function $f(s,h,\Phi_{i,l},\eta)$ in the $\GL_2$-case is defined in $\S$4, Eq.(\ref{eq45}). For quasi-characters $\chi, \eta_1$ of $F^\times$, we define $\eta_2=\omega_\pi^{-1}\chi^{-1}\eta_1^{-1}$. Let $l=\deg(\eta_1\eta_2^{-1})=\deg(\eta_1^2 \cdot\omega_\pi\cdot \chi)$, $\eta=(\eta_1,\eta_2)$ and $\eta^*=(\eta_2,\eta_1)$. It is easy to see that
\begin{equation}{\label{eq58*}}
f(\bar n(x), \hat \Phi_{i,l},\eta, s)=\left\{\begin{array}{lll} q_F^{-l}& \textrm{ if } |x|_F\le q_F^{-i}, \\  0, &\textrm{otherwise}  \end{array} \right.
\end{equation}
and
\begin{equation}{\label{eq59*}}
f(wn(x), \hat \Phi_{i,l},\eta^*, 1-s)=q_F^{-l-i}c(s,\eta,\psi), \textrm{ if } |x|_F\le q_F^{i-l},
\end{equation}
where $ c(s,\eta,\psi)=\int_{|y|_F\le q_F^l}\psi_F(y)\eta_1^{-1}\eta_2(y)|y|^{1-s} d^*y$, which has no zeros or poles except for a finite number of $q_F^{-s}$.

\begin{proof}[Proof of Theorem $\ref{thm511}$]
The proof is almost identical to the proof of Theorem \ref{thm39}.

We fix vectors $v\in V_\pi, v'\in V_{\pi'}$ such that $W_{v}(1)=1=W_{v'}(1)$. Let $L_1$ be an even integer such that $v,v'$ are fixed by $K_{L_1}$.  Let $L=3L_1$ and let  $m$ be an integer such that $L\le m$. Recall that we have defined $\phi^m\in \CS(F^2)$, the space of $\omega_{\psi^{-1}}$ in $\S$\ref{sec52}. Let $\chi$ be a character of $F^\times$ with $\deg(\chi)\le m$, recall that $\theta^{m,\chi}$ is the Whittaker function associated with $\phi^m_\chi\in \omega_{\psi^{-1},\chi}$, see Eq.(\ref{eq51}).
 
For a quasi-character $\eta_1$ of $F^\times$, we take $\eta_2$ such that  $\omega_\pi\cdot \chi\cdot \eta_1\eta_2=1$.  Let $l=\deg(\eta_1\eta_2^{-1})$.  We consider the integral
$$\Psi(s,W_{v_m},\theta^{m,\chi},\Phi_{i,l},\eta)=\int_{R\setminus H}W_{v_m}(h)\theta^{m,\chi}(h)f(s,h,\Phi_{i,l},\eta)dh.$$

We will take the integral $\Psi(s,W_{v_m}, \theta^{m,\chi},\Phi_{i,l})$ on the open dense set $R\setminus NT\bar N$ of $R\setminus H$. For $h=n(y)t(a,b)\bar n(x)\in NT\bar N$, the Haar measure will be $dh=|a/b|_F^{-1}dyd^*a dx$. Thus by Eq.(\ref{eq58*}), we get
\begin{align*}
&\qquad \Psi(s,W_{v_m},\theta^{m,\chi},\Phi_{i,l},\eta)\\
&=\int_{F^\times}\int_F W_{v_m}(t(a)\bar n(x))\theta^{m,\chi}(t(a)\bar n(x))\eta_1(a)|a|^{s-1}f(s,\bar n(x),\Phi_{i,l},\eta)dxd^*a\\
&=q_{F}^{-l}\int_{F^\times}\int_{|x|_F\le q_F^{-i}}W_{v_m}(t(a)\bar n(x))\theta^{m,\chi}(t(a)\bar n(x))\eta_1(a)|a|^{s-1}dxd^*a.
\end{align*}

 Let $i\ge 3m$, then $|x|_F\le q_F^{-i}$ implies that $\bar n(x)\in J_m$. Thus by Lemma \ref {lem51} and Proposition \ref{prop56}, we have 
$$W_{v_m}(t(a)\bar n(x))= W_{v_m}(t(a)),\theta^{m,\chi}(t(a)\bar n(x))=\theta^{m,\chi}(t(a)) \textrm{ for } |x|_F\le q_F^{-i}.$$
Combining these facts and Lemmas \ref{lem51}, \ref{lem52}, we get
\begin{align*}
&\quad \Psi(s,W_{v_m},\theta^{m,\chi},\Phi_{i,l})\\
&=q_F^{-l-i} \int_{F^\times}W_{v_m}(t(a))\theta^{m,\chi}(t(a))\eta_1(a)|a|^{s-1}d^*a\\
&=q_F^{-l-i}\int_{1+\CP_F^m}W_{v_m}(t(a))\theta^{m,\chi}(t(a)\eta_1(a)d^*a\\
&=q_F^{-l-i} \int_{1+\CP_F^m}\theta^{m,\chi}(t(a))\eta_1(a)d^*a.
\end{align*}
 By definition of $\theta^{m,\chi}$, see Eq.(\ref{eq51}), for $a\in 1+\CP_E^m$, we have 
\begin{align*}&\theta^{m,\chi}(t(a))\\
&=\int_{F^\times}\chi^{-1}(u)(\omega_{\psi^{-1}}(t(a))\phi^m)(u,u^{-1})d^*u\\
&=\int_{F^\times}\chi^{-1}(u)|a|^{1/2}\phi^{m}(au,u^{-1})d^*u\\
&=\int_{1+\CP_F^m}\chi^{-1}(u)d^*u\\
&=\vol(1+\CP_F^m, d^*u),
\end{align*}
where in the last step we used the fact that $\deg(\chi)\le m$.
Thus if $m\ge l(\eta_1)$, we have
\begin{align}
&\quad \Psi(s,W_{v_m},\theta^{m,\chi},\Phi_{i,l}) \nonumber\\
&=q_F^{-l-i} \vol(1+\CP_F^m, d^*a)\int_{1+\CP_F^m}\eta_1(a)d^*a \nonumber\\
&=q_F^{-l-i} \vol(1+\CP_F^m, d^*a)^2.\label{eq58}
\end{align}

The same calculation works for $\Psi(s,W_{v'_m},\theta^{m,\chi},\Phi_{i,l})$. In particular, we get
\begin{equation}{\label{eq59}}\Psi(s,W_{v_m'},\phi^{m,\chi},\Phi_{i,l},\eta)=\Psi(s,W_{v_m},\phi^{m,\chi},\Phi_{i,l})=q_F^{-l-i} \vol(1+\CP_F^m, d^*a)^2,\end{equation}
for $i\ge 3m, m\ge \max\wpair{L, l(\eta_1),\deg(\chi)}$.

Let $W=W_{v_m}$ or $W_{v_m'}$. We compute the integral $\Psi(1-s,W,\phi^{m,\chi}, \hat \Phi_{i,l})$ on the dense set $N\setminus NTwN$ of $N\setminus H$.  Let $i\ge m+l $, then $|x|_F\le q_E^{m}$ implies that $|x|_F\le q^{i-l}_F$. Then by Eq.(\ref{eq59*}), we have
\begin{align}
&\quad\Psi(1-s,W,\phi^{m,\chi}, \hat \Phi_{i,l}) \nonumber\\
&=\int_{F^\times}\int_{F }W(t(a)wn(x))\theta^{m,\chi}(t(a)wn(x)) f_s(t(a)wn(x), \hat \Phi_{i,l}, \eta^*)\nonumber\\
&=q_F^{-l-i}c(1-s, \eta, \psi)\int_{F^\times}\int_{|x|\le q_F^m}W(t(a)wn(x) )\theta^{m,\chi}(t(a)wn(x))\eta_2(a)|a|^{-s}dxd^*a \label{eq510}\\
&+\int_{F^\times}\int_{|x|_F>q_F^m} W(t(a)wn(x) )\theta^{m,\chi}(t(a)wn(x)) f(1-s, t(a)wn(x), \hat \Phi_{i,l})dxd^*a \label{eq511}
\end{align}
By Proposition \ref{prop53}, we have $W_{v_m}(t(a)wn(x))=W_{v_m'}(t(a)wn(x))$ for $|x|>q_F^m$. Thus the expressions (\ref{eq511}) for $W=W_{v_m}$ and $W=W_{v'_m}$ are the same.

 For $|x|_F\le q_F^m$, we have $n(x)\in N_m$. By Lemma \ref{lem51} and Proposition \ref{prop56},  we have
$$W_{v_m}(t(a)wn(x))=\psi(x)W_{v_m}(t(a)w), \textrm{ and }\theta^{m,\chi}(t(a)wn(x))=\psi^{-1}(x)\omega^{m,\chi}.$$

Thus the expression (\ref{eq510}) can be simplified to
\begin{equation}{\label{eq512}}q_F^{-l-i+m}c(1-s, \eta, \psi)\int_{F^\times}W(t(a)w)\theta^{m,\chi}(t(a)w)\eta_2(a)|a|^{-s}d^*a.  \end{equation}
Thus we get
\begin{align}
&\Psi(1-s,W_{v_m},\phi^{m,\chi},\hat \Phi_{i,l})-\Psi(1-s,W_{v_m'},\phi^{m,\chi},\hat \Phi_{i,l}) \label{eq513}\\
=&q_F^{-l-i+m}c(1-s, \eta^*, \psi)\int_{F^\times}(W_{v_m}(t(a)w)-W_{v_m'}(t(a)w))\theta^{m,\chi}(t(a)w)\eta_2(a)|a|^{-s}d^*a\nonumber
\end{align}
By Eq.(\ref{eq59}), Eq.(\ref{eq513}) and the local functional equation, we get
\begin{align}
&d_{m}(s,\eta,\psi)(\gamma(s,\pi, \omega_{\psi^{-1},\chi}, \eta)-\gamma(s,\pi',\omega_{\psi^{-1}},\eta)) \label{eq514}\\
=&\int_{F^\times}(W_{v_m}(t(a)w)-W_{v_m'}(t(a)w))\theta^{m,\chi}(t(a)w)\eta_2(a)|a|^{-s}d^*a \nonumber
\end{align}
for $m\ge\max\wpair{L, l(\eta_1)}$, where $d_m(s,\eta,\psi)=\vol(1+\CP_F^m,d^*a)^2q_F^{-m}c(1-s,\eta,\psi)^{-1}$.
By Proposition \ref{prop53}, for $n\in N_m-N_L$, we have $W_{v_L}(t(a)wn)=W_{v_L'}(t(a)wn)$. Thus by Lemma (\ref{lem51}), we have
\begin{align}
&\quad W_{v_m}(t(a)w)-W_{v_m'}(t(a)w)\label{eq515}\\
&=\frac{1}{\vol(N_m)}\int_{N_m}\psi^{-1}(n)(W_{v_L}(t(a)wn)-W_{v'_L}(t(a)wn))dn \nonumber\\
&=\frac{1}{\vol(N_m)}\int_{N_L}\psi^{-1}(n)(W_{v_L}(t(a)wn)-W_{v'_L}(t(a)wn))dn \nonumber\\
&=\frac{\vol(N_L)}{\vol(N_m)}(W_{v_L}(t(a)w)-W_{v_L'}(t(a)w)).\nonumber
\end{align}

If we combine Eq.(\ref{eq514}) and Eq.(\ref{eq515}), we get
\begin{align}
&d_{m,L}(s,\eta,\psi)(\gamma(s,\pi,\omega_{\psi^{-1},\chi},\eta)-\gamma(s,\pi',\omega_{\psi^{-1},\chi},\eta)) \label{eq516}\\
=&\int_{F^\times}(W_{v_L}(t(a)w)-W_{v_L'}(t(a)w))\theta^{m,\chi}(t(a)w)\omega_\pi^{-1}\chi^{-1}\eta_1^{-1}(a)|a|^{-s}d^*a,\nonumber
\end{align}
with $d_{m,L}(s,\eta,\psi)=d_m(s,\eta,\psi)\vol(N_m)\vol(N_L)^{-1}$, for $m\ge \max\wpair{L,l(\eta_1)}$.

Now we are ready to prove our theorem. To prove (1), it suffices to show that $W_{v_L}(h)=W_{v_L'}(h)$ for all $h\in H=\GL_2(F)$. By Proposition \ref{prop53}, it suffices to show that $W_{v_L}(t(a)w)=W_{v_L'}(t(a)w)$ for all $a\in E^\times$. Let $C=3L/2>L$. By Lemma \ref{lem52}, we have 
$$W_{v_L}(t(a)w)=W_{v'_L}(t(a)w)=0, \textrm{ if }a\notin \CP_F^{-C}.$$
Thus it suffices to show that 
$$W_{v_L}(t(a)w)-W_{v_L'}(t(a)w)=0, \forall a\in E^\times, \textrm{ with }|a|_F\le q_F^C. $$

By our assumption on $\gamma$-factors and Eq.(\ref{eq516}), we have
$$\int_{F^\times}(W_{v_L}(t(a)w)-W_{v_L'}(t(a)))\theta^{m,\chi}(t(a)w)\eta_2(a)|a|^{-s}d^*a=0. $$
By the inverse Mellin transform, we get
$$ (W_{v_L}(t(a)w)-W_{v_L'}(t(a)w))\theta^{m,\chi}(t(a)w), \forall a\in F^\times.$$

We take $m\ge \max\wpair{C,l(\eta_1), \deg(\chi)}$, then by Proposition \ref{prop58}, for each $a\in \CP_F^{-C}$, there exists a character $\chi\ne |~|^{\pm1}$ with $\deg(\chi)\le C$ such that $\theta^{m,\chi}(t(a)w)\ne 0$. Thus we get 
$$ W_{v_L}(t(a)w)=W_{v_L'}(t(a)w, \forall a\in \CP_F^{-C}.$$
This completes the proof of (1). 

To prove (2), we take an integer $l_1$ such that 
$$W_{v_L}(t(a_0a)w)=W_{v_L}(t(a_0a)w),  \textrm{ and }W_{v_L'}(t(a_0a)w)=W_{v_L'}(t(a_0a)w),$$
for all $a_0\in 1+\CP_F^{l_1}$. Note that $l_1$ only depends on $\pi,\pi'$.  By Corollary (\ref{cor59}), for $m\ge C$, we have 
$$\theta^{m,\chi}(t(a_0a)w)=\theta^{m,\chi}(t(a)w), \forall a_0\in 1+\CP^{\deg(\chi)+C}.$$
Let $l_0=\max\wpair{\deg(\chi)+C, l_1, \deg(\omega_\pi), \deg(\chi)}$. Note that $l_1, L, C=3L/2$ are only depend on $\pi,\pi'$, thus $l_0$ only depends on $\pi,\pi',\chi$. Then it is clear that for $\deg(\eta_1)>l_0$, the right side of Eq.(\ref{eq516}) vanishes. Since $d_{m,L}(s,\eta,\psi)$ does not have zeros or poles other than finite number of $q_F^{-s}$, we get 
$$ \gamma(s,\pi,\omega_{\psi^{-1},\chi},\eta_1)=\gamma(s,\pi',\omega_{\psi^{-1},\chi},\eta_1), \textrm{ if } \deg(\eta_1)>l_0.$$
This completes the proof of (2).
\end{proof}

\end{document}